\documentclass{amsart}

\usepackage{amsmath,amssymb,amsthm}
\usepackage{hyperref}
\usepackage{xcolor}

\newtheorem{thmintr}{Theorem}

\newtheorem{prop}{Proposition}[subsection]
\newtheorem{thm}[prop]{Theorem}
\newtheorem{lem}[prop]{Lemma}
\newtheorem{cor}[prop]{Corollary}

\theoremstyle{definition}

\newtheorem{rem}[prop]{Remark}
\newtheorem{exwith}[prop]{Example}

\newtheorem*{ack}{Acknowledgement}

\def\co{\colon\thinspace}

\newcommand{\AAA}{\mathcal A}

\newcommand{\CC}{\mathcal C}

\newcommand{\rmd}{\mathrm d}

\newcommand{\DD}{\mathcal D}

\newcommand{\rme}{\mathrm e}

\newcommand{\EE}{\mathcal E}

\newcommand{\FF}{\mathcal F}

\newcommand{\II}{\mathcal I}

\newcommand{\MM}{\mathcal M}

\newcommand{\OO}{\mathcal O}

\newcommand{\R}{\mathbb R}

\renewcommand{\SS}{\mathcal S}

\newcommand{\UU}{\mathcal U}

\newcommand{\VV}{\mathcal V}

\newcommand{\bfx}{\mathbf x}
\newcommand{\XX}{\mathcal X}

\newcommand{\bfy}{\mathbf y}

\newcommand{\lra}{\longrightarrow}
\newcommand{\ra}{\rightarrow}

\DeclareMathOperator{\Diff}{\mathrm{Diff}}

\DeclareMathOperator{\End}{\mathrm{End}}

\DeclareMathOperator{\id}{\mathrm{id}}

\DeclareMathOperator{\Int}{\mathrm{Int}}

\DeclareMathOperator{\st}{\mathrm{st}}

\DeclareMathOperator{\supp}{\mathrm{supp}}

\begin{document}

\author{Florian Buck}
\author{Christopher Schmidt}
\author{Kai Zehmisch}
\address{Fakult\"at f\"ur Mathematik, Ruhr-Universit\"at Bochum,
Universit\"atsstra{\ss}e 150, D-44801 Bochum, Germany}
\email{Florian.Buck@rub.de}
\email{Christopher.Schmidt-y8n@rub.de}
\email{Kai.Zehmisch@rub.de}

\title[A contact-geometric variant of Smale's contraction]
{A contact-geometric variant of Smale's contraction}

\date{06.09.2026}

\begin{abstract}
Motivated by Smale's contraction
of the compactly supported diffeomorphism group
of the two-disc,
we study contact forms
on an open rotationally symmetric Darboux ball
that agree with the standard contact form
outside a compact set
and whose Reeb flows have no trapped orbits.
Such forms are called vertically convex.

In dimensions three and five,
we identify the quotient of their space
by compactly supported diffeomorphisms
with a contractible monodromy space
and prove an equivariant product splitting.
Consequently,
the orbit of the standard contact form
is a strong deformation retract.
It follows that the space
of vertically convex contact forms
is contractible in dimension three
and homotopy equivalent
to the compactly supported diffeomorphism group,
hence connected,
in dimension five.

The analogous subspace of forms
defining the standard contact structure
has the homotopy type
of the corresponding compactly supported
contactomorphism group in dimension five
and is contractible in dimension three.
\end{abstract}


\maketitle


\section{Introduction\label{sec:intro}}

Smale's contraction \cite{sm59}
of the compactly supported diffeomorphism group
of the open two-disc
may be interpreted in terms of a space of vector fields
whose trajectories are not trapped.
The present paper develops
a contact-geometric variant of this point of view.
We replace such vector fields by contact forms
whose Reeb flows have no trapped orbits
and study their space
through the induced symplectic monodromy.
This leads to an equivariant product splitting
and determines the homotopy type
of the resulting spaces
in dimensions three and five.

The boundary condition underlying
this class of contact forms
was introduced by Eliashberg and Hofer \cite{eh94}.
Let $(M,\xi)$ be a connected, compact, cooriented
$(2n+1)$-dimensional contact manifold with boundary.
They require a contact embedding
of a neighbourhood $U$ of $\partial M$
into $\R\times\R^{2n}$,
called the {\bf gluing map},
such that $\partial M$ is mapped
onto the unit sphere
$S^{2n}=\partial D^{2n+1}$
and $U$ into the closed unit ball $D^{2n+1}$.
Here,
$\R\times\R^{2n}$ is equipped with
the standard contact structure $\xi_{\st}$
defined by
\[
\alpha_{\st}:=
\rmd b+\tfrac12(\bfx\rmd\bfy-\bfy\rmd\bfx)
\,,
\]
where $b\in\R$ denotes the vertical coordinate.
The gluing map allows one to choose
a $\xi$-defining contact form $\alpha$
whose restriction to $U$ is the pullback
of $\alpha_{\st}$
and to complete $(M,\alpha)$
by attaching the exterior of $D^{2n+1}$.
The resulting contact form agrees with
$\alpha_{\st}$ outside a compact set.
A Reeb orbit through $M$ is called trapped
in forward or backward time
if it does not leave $M$ through the boundary
in the corresponding time direction.
The form $\alpha$ is called {\bf vertically convex}
if its completed Reeb flow has no trapped orbits.
The contact structure $\xi$ is called
{\bf vertically convex}
if it admits such a defining contact form.

For $n=1$,
Eliashberg and Hofer \cite{eh94}
proved by means of holomorphic-disc fillings
that $\xi$ is vertically convex
if and only if it is {\bf hypertight},
meaning that it admits a defining contact form
without contractible periodic Reeb orbits.
In this situation,
\cite[Theorem~1]{eh94} implies that
$M$ is diffeomorphic to $D^3$.
In the subsequent remarks and questions,
Eliashberg and Hofer further state that
$(M,\rmd\alpha)$ is odd-symplectomorphic to
$(D^3,\rmd x\wedge\rmd y)$
and, using the uniqueness
of the tight contact structure on $D^3$
\cite{elia92},
that $(M,\xi)$ is contactomorphic to
$(D^3,\xi_{\st})$;
see \cite[p.\ 1305, Items~1.a) and 1.b)]{eh94}.
These conclusions are presented there
as natural extensions of the method developed in the paper,
although the corresponding arguments are not spelled out.
We recover them below in
Proposition \ref{prop:oddsymplnotnecext}
and Corollary \ref{cor:roundball}.

In higher dimensions,
hypertightness and vertical convexity
are no longer equivalent.
Indeed, \cite{grz14} constructs contact forms
defining $\xi_{\st}$ on $D^{2n+1}$
whose Reeb vector fields agree with $\partial_b$
near the boundary
and have trapped but no periodic orbits.
These forms are hypertight
but not vertically convex.

Eliashberg and Hofer also announced
a higher-dimensional analogue
of their three-dimensional result
under the additional assumption
$H_2(M;\R)=0$;
see \cite[p.\ 1305, Item~2]{eh94}.
The details of the argument
were not included in \cite{eh94}.
The announced result was subsequently proved
without this homological assumption
in \cite[Theorem~1]{gz16b}.
In particular,
under the standard boundary condition,
hypertightness implies that
$M$ is diffeomorphic to $D^{2n+1}$
in all dimensions.
The underlying degree method
for holomorphic discs
was subsequently extended
to compact contact manifolds with boundary
modelled on more general shapes
in \cite{bschz19,beck23,kwz22}.

For the formulation of our main results,
denote by $B^m=\Int(D^m)$
the open unit ball,
by $\DD^m$ the group of compactly supported
diffeomorphisms of $B^m$
and by $\SS^{2n}\subset\DD^{2n}$
the subgroup of symplectomorphisms of
$(B^{2n},\rmd\bfx\wedge\rmd\bfy)$.
We equip $\DD^m$ with the strong
$C^{\infty}$-topology described in
Section \ref{subsec:functionspaces}
and $\SS^{2n}\subset\DD^{2n}$
with the induced topology.
The groups $\SS^2$ and $\SS^4$
are contractible by
\cite[Theorem~1.8]{agz22}
and Gromov's result \cite{grom85}, resp.
We study the space $\AAA^{2n+1}$
of vertically convex contact forms
on $B^{2n+1}$
that agree with $\alpha_{\st}$
outside a compact set,
equipped with the strong $C^{\infty}$-topology.
While \cite{efz} classifies rotationally symmetric
vertically convex domains
by contact-geometric invariants,
the present paper fixes the shape
and studies monodromy and the topology
of the resulting space of contact forms.

Following the Reeb flow
from the lower to the upper part of the boundary
associates with each element of $\AAA^{2n+1}$
a compactly supported exact symplectomorphism
of $B^{2n}$,
its {\bf monodromy},
together with an action function.
We denote by $\EE$ the space of all pairs
$(\varphi,f)\in\SS^{2n}\times C_c^{\infty}(B^{2n})$
satisfying
$\varphi^*\lambda_{\st}=\lambda_{\st}-\rmd f$,
where 
\[
\lambda_{\st}:=\tfrac12(\bfx\rmd\bfy-\bfy\rmd\bfx)
\,.
\]
The pairs $(\varphi,f)$ in $\EE$
that arise as monodromy data
are precisely those satisfying
the explicit {\bf shape inequality}
\[
-\sqrt{1-|\bfx|^2-|\bfy|^2}
<
\varphi^*\!\left(
\sqrt{1-|\bfx|^2-|\bfy|^2}
\right)-f
\,.
\]
They form the monodromy space $\MM$.
The monodromy homeomorphism
of Proposition \ref{prop:monodromyhomeo}
identifies the quotient
$\AAA^{2n+1}/\DD^{2n+1}$
with $\MM$.
Writing $\EE_0$ for the path component
of $(\id_{B^{2n}},0)$ in $\EE$
and writing $\MM_0$ for $\MM\cap\EE_0$,
Proposition \ref{prop:homoequiv}
shows that the inclusion
\[
\MM_0\lhook\joinrel\longrightarrow\EE_0
\]
is a homotopy equivalence.
For $n=1,2$,
Proposition \ref{prop:mscontractible} shows
that the monodromy space $\MM$ itself is contractible.
The push-forward action of $\DD^{2n+1}$
on $\AAA^{2n+1}$ is free
by Lemma \ref{lem:freeaction}.
Using time-$C^1$ contractions
of $\SS^2$ and $\SS^4$,
we construct a continuous global section
of the quotient projection.
The equivariant splitting
of Theorem \ref{thm:splitting} then gives
an $\DD^{2n+1}$-equivariant homeomorphism
\[
\AAA^{2n+1}
\cong
\DD^{2n+1}\times
\big(
\AAA^{2n+1}/\DD^{2n+1}
\big)
\,,
\qquad
\text{for } n=1,2
\,.
\]
In fact,
Corollary \ref{cor:orbitdeformationretracts}
shows that the orbit
$\DD^{2n+1}\cdot\alpha_{\st}$
is a strong deformation retract
of $\AAA^{2n+1}$.
Consequently,
the orbit map
\[
\DD^{2n+1}\lra\AAA^{2n+1}
\,,
\qquad
\Phi\longmapsto\Phi_*\alpha_{\st}
\,,
\qquad
\text{for } n=1,2
\,,
\]
is a homotopy equivalence.

Notice that the space $\mathcal A(M)$
appearing in \cite[p.\ 1305, Item~1.c)]{eh94}
is formally defined by the condition
$\inf(\alpha)>\pi$.
For $M=D^3$,
the holomorphic-disc foliation obtained in \cite{eh94}
shows that this condition
is equivalent to vertical convexity.
After restricting forms that are standard
near $\partial D^3$ to the open ball,
$\mathcal A(D^3)$ therefore identifies naturally
with $\AAA^3$.
As $\DD^3$ is contractible
by Hatcher's result \cite{hat83},
Corollary \ref{cor:orbitdeformationretracts}
proves the contractibility assertion
announced in
\cite[p.\ 1305, Item~1.c)]{eh94}.
More generally,
the main conclusions are:

\begin{thmintr}
\label{thmintr:homotopytypeofaaa}
The following homotopy equivalences
  \[
  \AAA^3\simeq*
  \quad\text{and}\quad
  \AAA^5\simeq\DD^5
  \]
hold true.
In particular, $\AAA^5$ is connected.
\end{thmintr}

Theorem \ref{thmintr:homotopytypeofaaa}
may be viewed as a contact-geometric interpretation
of a setup underlying Smale's proof \cite{sm59}
of a time-$C^1$ contraction of $\DD^2$.
In Theorem \ref{thm:homotopytypeofxx3}
we will formulate a smooth analogue of
Theorem \ref{thmintr:homotopytypeofaaa},
which inspired this work.
The following theorem is a variant of
Theorem \ref{thmintr:homotopytypeofaaa}
for the subgroup
\[
\DD^{2n+1}_{\xi_{\st}}\subset\DD^{2n+1}
\]
given by contactomorphisms
of the standard contact ball $(B^{2n+1},\xi_{\st})$
and the subspace
\[
\AAA^{2n+1}_{\xi_{\st}}\subset\AAA^{2n+1}
\]
that consists of $\xi_{\st}$-defining contact forms:

\begin{thmintr}
\label{thmintr:homotopytypeofaaacont}
The following homotopy equivalences
  \[
  \AAA_{\xi_{\st}}^3\simeq*
  \quad\text{and}\quad
  \AAA_{\xi_{\st}}^5\simeq\DD_{\xi_{\st}}^5
  \]
hold true.
\end{thmintr}

The proof is given in Section \ref{sec:spofcontforms}.
Notice that $\DD_{\xi_{\st}}^3$ is contractible
as a consequence of the parametric result
of Eliashberg--Mishachev \cite{em}.

For a fixed open book decomposition,
D\"orner \cite{doe14} showed that in dimension three
every component of the space of adapted contact forms
is contractible and constructed higher-dimensional
obstructions to homotopies through adapted contact forms.
Related quotients of spaces of contact forms
defining a fixed contact structure
by the identity component of the contactomorphism group
were considered by Polterovich
\cite[Section~3]{pol02}
in connection with contact dissipation.
In our setting,
monodromy provides an explicit global parametrisation
of the quotient.

We remark that the $\xi_{\st}$-variant of
the splitting Theorem \ref{thm:splitting},
which will be used in
Theorem \ref{thmintr:homotopytypeofaaacont},
can be applied to the following problem
in dimensions $3$ and $5$,
see Remark \ref{rem:non-representability}:
Given $h\in C_c^{\infty}(B)$ --
suppressing the dimension superscript  --
such that $\rme^h\alpha_{\st}\in\AAA_{\xi_{\st}}$,
find a (necessarily unique) solution $\Phi\in\DD_{\xi_{\st}}$
of the equation
\[
\Phi_*\alpha_{\st}=\rme^h\alpha_{\st}
\,.
\]
A solution exists precisely when
$\rme^h\alpha_{\st}$ is contained in the
$\DD_{\xi_{\st}}$-orbit of $\alpha_{\st}$,
which is homeomorphic to a copy of $\DD_{\xi_{\st}}$
in the product of $\DD_{\xi_{\st}}$ with
$\AAA_{\xi_{\st}}/\DD_{\xi_{\st}}$
for $n=1,2$.
The quotient is homeomorphic to an open subset
of $\SS$ for $n=1,2$ by
Propositions \ref{prop:mscontractible}
and \ref{prop:monodromyhomeo}.

Besides these homotopy-theoretic results,
we obtain several further geometric consequences.
Motivated by the considerations in \cite{beck23,kschz20}
we will investigate vertical convexity systematically.
We show in Proposition \ref{prop:oddsymplnotnecext}
and Corollary \ref{cor:oddsymlifnotrappedorbits}
that the underlying odd-symplectic manifold $(M,\rmd\alpha)$
of the vertically convex strict contact manifold $(M,\alpha)$
shaped by $S^{2n}$ is odd-symplectomorphic to
$(D^{2n+1},\rmd\bfx\wedge\rmd\bfy)$ in all dimensions.
On the other hand,
Proposition \ref{prop:inwardextension}
and Remark \ref{rem:exotic}
show that, whenever $\Gamma_{2n+1}$ is trivial,
the gluing map extends after restriction
to a diffeomorphism $M\ra D^{2n+1}$.
Here, $\Gamma_m$ denotes the group
of twisted oriented $m$-spheres,
with $m$ referring to the dimension of the disc $D^m$
over which diffeomorphisms of $S^{m-1}$
are to be extended.
In Corollary \ref{cor:roundball}
we will show that connectedness of $\SS^2$
and an elementary integration
yield the contact uniqueness result
from \cite[p.\ 1305, Item~1.b)]{eh94} for $D^3$,
whose historical approach we will present
in Remark \ref{rem:alternative}.
Furthermore, for $n=2$, we will show
in Corollary \ref{cor:roundball}
that the vertically convex contact manifold $(M,\xi)$
is contactomorphic to $(D^5,\xi_{\st})$
using connectedness of $\SS^4$.


\section{Defining the geometry}
\label{sec:defthegeom}

We introduce vertically convex contact geometry,
see also \cite{beck23,efz,kschz20}.


\subsection{Shapes}

Let $V_S$ be an open subset of a smooth manifold $V$.
Consider smooth functions
$f_{\pm}\co V_S\ra\R$ with $f_-<f_+$
that extend continuously to the closure of $V_S$.
We use the same notation for their extensions
and assume that $f_-=f_+$ along $\partial V_S$.
Assume that the closure of the union
of the graphs $S^{\pm}$ of $f_{\pm}\co V_S\ra\R$
is equal to a compact, smooth hypersurface $S$ in $\R\times V$,
which we call {\bf strict vertically convex}.
In particular, $V_S$ is relatively compact.

We call $S^-$ the {\bf entrance set}, $S^+$ the {\bf exit set}
and $V_S$ the {\bf shadow set} of $S$.
We call the graph of $f_{\pm}|_{\partial V_S}$,
which equals $E_S:=\partial S^{\pm}$,
the {\bf equatorial set} of $S$.
If $E_S$ is a smooth submanifold
we will simply call it the {\bf equator}.
The {\bf domain of} $S$ is the closure $D_S$
in $\R\times V$ of the set of all points
$(b,v)\in\R\times V_S$
such that $f_-(v)\leq b\leq f_+(v)$.
The {\bf end} of $D_S$ is
\[
\End(D_S):=
(\R\times V)\setminus\Int(D_S)
\,.
\]
The triple $(S,V_S,f_{\pm})$ or $S$ itself
is called a {\bf shape}.
We call two shapes $(S,V_S,f_{\pm})$ and $(T,V_T,g_{\pm})$
{\bf equivalent} provided that $V_S=V_T$ and that
the functions $f_+-g_+$ and $f_--g_-$ have support in $V_S$,
being necessarily compact in $V$,
see \cite{efz}.

Let $\lambda$ be a $1$-form on $V$
such that $\rmd\lambda$ is symplectic.
The cooriented contact structure of the
{\bf contactisation} $(\R\times V,\rmd b+\lambda)$
is denoted by $\xi_{\lambda}$.
Necessary and sufficient conditions for the domains
$(D_S,\xi_{\lambda})$ and $(D_T,\xi_{\lambda})$
of equivalent shapes $(S,V_S,f_{\pm})$ and $(T,V_T,g_{\pm})$
to be contactomorphic up to the boundary
are found in \cite{efz} for subclasses of shapes.


\subsection{Shaped contact manifolds}
\label{subsec:shapedcontactmanifolds}

Let $M$ be a compact manifold with boundary
that admits a cooriented contact structure $\xi$.
We call the contact manifold $(M,\xi)$ {\bf shaped}
provided that there exist a shape $(S,V_S,f_{\pm})$
in $(\R\times V,\rmd b+\lambda)$
and a contact embedding
\[
\psi\co
\big(U,\partial M,\xi\big)
\lra
\big(D_S,S,\xi_{\lambda}\big)
\]
of an open neighbourhood $U$ of $\partial M$ in $M$.
In particular the boundary of $M$ is non-empty.
The embedding $\psi$ is called a {\bf gluing map}.
The set of $\xi$-defining contact forms $\alpha$ on $M$
such that $\psi^*(\rmd b+\lambda)=\alpha$
on an open neighbourhood of $\partial M$
is non-empty.
We call the corresponding strict contact manifolds $(M,\alpha)$ {\bf shaped}.
The {\bf completion}
\[
(\hat{M},\hat{\alpha}):=
(M,\alpha)\cup_{\psi}
\big(\!\End(D_S),\rmd b+\lambda\big)
\]
of $(M,\alpha)$ and $\psi$
is obtained by gluing $(M,\alpha)$
to $\End(D_S)$ via $\psi$.
The contact form $\hat{\alpha}$ restricts
to $\alpha$ on $M$ and to $\rmd b+\lambda$ on
$\hat{M}\setminus\Int(M)=\End(D_S)$.
Set $\hat{\xi}:=\ker\hat{\alpha}$.

Vertical convexity of $S$ ensures
that completing $(M,\alpha)$
does not create any closed Reeb orbits of $\hat{\alpha}$
other than the closed Reeb orbits of $\alpha$ on $M$.
A Reeb orbit of $(\hat{M},\hat{\alpha})$ through $\Int(M)$
is called {\bf trapped in forward time}
if its positive half-orbit does not meet $S^+$,
and {\bf trapped in backward time}
if its negative half-orbit does not meet $S^-$.
All periodic Reeb orbits are trapped.
A shaped contact manifold $(M,\xi)$
is called {\bf vertically convex}
provided that there exists a shaped, $\xi$-defining
contact form $\alpha$ such that
$\hat{\alpha}$ admits no trapped Reeb orbits.
The strict contact manifold $(M,\alpha)$
is also called {\bf vertically convex}.
Observe that a vertically convex contact manifold
does not admit connected components
without boundary.

In view of \cite[p.\ 1305, Item~1.a)]{eh94} we formulate:

\begin{prop}
 \label{prop:oddsymplnotnecext}
 Let $(M,\alpha)$ be a vertically convex strict contact manifold
 shaped by $(S,V_S,f_{\pm})$ in $(\R\times V,\rmd b+\lambda)$.
 Then $(M,\rmd\alpha)$ is odd-symplectomorphic
 to $(D_S,\rmd\lambda)$.
 In particular,
 the boundary of each connected component of $M$
 is non-empty and connected.
\end{prop}

The proof is given in Remark
\ref{rem:oddsymplnotnecext};
an alternative construction is provided in
Section \ref{subsec:oddsympuniness}.

Following the Reeb flow
of a vertically convex strict contact manifold
from $S^-$ to $S^+$ defines a diffeomorphism.
The second component w.r.t.\ the splitting $\R\times V$,
the so-called {\bf monodromy} of $(M,\alpha)$,
is a compactly supported,
exact symplectomorphism of the shadow set $(V_S,\lambda)$,
see Proposition \ref{prop:globalflowbox}.
A vertically convex contact manifold $(M,\xi)$
is called {\bf Hamiltonian} provided that there exists
a vertically convex $\xi$-defining contact form $\alpha$,
whose monodromy is a compactly supported
Hamiltonian diffeomorphism;
the strict contact manifold $(M,\alpha)$
is also called {\bf Hamiltonian}.
If the closure of the shadow set $(V_S,\lambda)$
of a shape $(S,V_S,f_{\pm})$ is a Liouville domain
and $f_{\pm}|_{\partial V_S}=0$,
then we will call $(S,V_S,f_{\pm})$ a {\bf Liouville shape}.
In Section \ref{subsec:changingtheaction} we will prove:

\begin{thm}[Contact uniqueness]
 \label{thm:monodromy}
 If a vertically convex contact manifold $(M,\xi)$
 shaped by a Liouville shape $(S,V_S,f_{\pm})$
 in $(\R\times V,\rmd b+\lambda)$ is Hamiltonian,
 then there exists a contactomorphism
 $(M,\xi)\ra\big(D_S,\xi_{\lambda}\big)$
 that restricts to the gluing map
 $\psi\co\big(U,\partial M,\xi\big)\ra\big(D_S,S,\xi_{\lambda}\big)$
 of $(M,\xi)$ in a neighbourhood of $\partial M$.
\end{thm}


\subsection{Function spaces}
\label{subsec:functionspaces}

The shadow set $V_S$
of a shape $(S,V_S,f_{\pm})$ in $\R\times V$
is relatively compact in the ambient manifold $V$.
Let $W$ be a smooth manifold and fix a smooth base map
$\gamma_0\co V\ra W$.
We equip $C^{\infty}(V,W)$
with the {\bf strong} $C^{\infty}$-{\bf topology}.
Following \cite[Chapter~2, p.~35]{hir76},
a basic strong $C^r$-neighbourhood of
$\gamma\in C^{\infty}(V,W)$ is specified by
a locally finite family of source charts $(\phi_i,V_i)$,
compact sets $K_i\subset V_i$,
target charts $(\psi_i,W_i)$ containing $\gamma(K_i)$
and positive numbers $\varepsilon_i$.
It consists of all maps $\delta$ taking $K_i$ into $W_i$
and satisfying
\[
\big\|
D^k(\psi_i\circ\gamma\circ\phi_i^{-1})(x)
-D^k(\psi_i\circ\delta\circ\phi_i^{-1})(x)
\big\|<\varepsilon_i
\]
for all $i$, all $x\in\phi_i(K_i)$
and $k=0,\ldots,r$.
Allowing all finite $r$ gives the strong $C^{\infty}$-topology,
see \cite[Chapter~2, p.~36]{hir76}.
For each finite $r$, the strong $C^r$-topology
also admits an intrinsic description through
the $r$-jet prolongation
\[
j^r\co C^r(V,W)\ra C^0\big(V,J^r(V,W)\big)
\,.
\]
Indeed,
it is induced from the strong topology
on the space of continuous maps,
see \cite[Chapter~2, p.~62]{hir76}.
Since the strong topology is finer
than the corresponding weak topology,
strong $C^{\infty}$-convergence implies
uniform convergence of the maps
and, in local coordinates, of all their derivatives
on compact subsets,
see \cite[Chapter~2, pp.~58--62]{hir76}.

With the base map $\gamma_0$ understood, set
\[
C_c^{\infty}(V_S,W):=
\big\{
\gamma\in C^{\infty}(V,W)
\mid
\gamma=\gamma_0\text{ on a neighbourhood of }V\setminus V_S
\big\}
\,.
\]
The {\bf support} of $\gamma$ relative to $\gamma_0$ is
\[
\supp_{\gamma_0\!}(\gamma):=
\overline{\{v\in V\mid\gamma(v)\ne\gamma_0(v)\}}
\,,
\]
where the closure is taken in $V$.
Thus, $\gamma\in C_c^{\infty}(V_S,W)$ precisely when
$\supp_{\gamma_0\!}(\gamma)$ is a compact subset of $V_S$.
We then say that $\gamma$ is
{\bf compactly supported} in $V_S$.
More specifically, for a compact set $K\subset V_S$,
we say that $\gamma$ is supported in $K$ if
$\supp_{\gamma_0\!}(\gamma)\subset K$.
For $W=V$ and $\gamma_0=\id_V$
this convention applies to diffeomorphisms.
If $W\ra V$ is a smooth vector bundle,
we use the subspace of smooth sections and take
$\gamma_0$ to be the zero section.
Affine spaces of sections are reduced to this case
by subtracting their fixed base section.
All spaces of maps, diffeomorphisms, functions,
forms and vector fields below
are understood as subspaces of these ambient mapping spaces.

We equip $C_c^{\infty}(V_S,W)$ and all its subspaces
with the topology induced from $C^{\infty}(V,W)$
and again call it the {\bf strong} $C^{\infty}$-{\bf topology}.
This convention is extrinsic:
the maps are extended by $\gamma_0$ and compared on $V$,
not regarded merely as maps on the open manifold $V_S$.
Since $\bar{V}_S$ is compact,
only finitely many members of any locally finite family
occurring in a strong neighbourhood
can meet $\bar{V}_S$;
outside this set all maps under consideration
agree with $\gamma_0$.
Consequently, the induced topology is equivalently described
by uniform convergence of the $r$-jets on $\bar{V}_S$
for every finite $r$.
Choosing an auxiliary metric on each finite jet bundle,
the corresponding uniform jet distances form
a countable defining family;
hence this topology is metrisable.
Accordingly, a sequence may converge without its supports
lying in a common compact subset of $V_S$:
the supports may approach $\partial V_S$,
provided that the $r$-jets of the maps converge uniformly
on $\bar{V}_S$ for every finite $r$.

Throughout, the strong $C^{\infty}$-topology
for spaces associated with $D_S\subset\R\times V$
is defined analogously to the one for $V_S\subset V$.

A closely related distinction is made explicit
by Cieliebak and Eliashberg in \cite[Section~6.3, p.~37]{ce14}.
For pseudo-isotopies on $\R\times M$
they use uniform $C^{\infty}$-convergence on the whole cylinder,
rather than uniform $C^{\infty}$-convergence on compact subsets,
and distinguish the two by translating
a non-trivial pseudo-isotopy.
Their example illustrates why the ambient convention
must be stated when the source is non-compact.

\begin{exwith}[Approaching the boundary]
Let $V=W=\R$, $V_S=(-1,0)$ and $\gamma_0=0$.
Choose a non-zero function
$\chi\in C_c^\infty((-2,-1))$
and, for $\nu\geq3$, set
\[
\gamma_\nu(s):=
\rme^{-\nu}\chi(\nu s)
\,.
\]
Then
$\supp(\gamma_\nu)\subset
\left(-\frac{2}{\nu},-\frac{1}{\nu}\right)
\subset V_S$.
Thus, the supports approach $\partial V_S$
and are not contained in any common compact subset of $V_S$.
On the other hand, for every $k\geq0$,
\[
\big\|\gamma_\nu^{(k)}\big\|
\leq
\rme^{-\nu}\nu^k\,
\big\|\chi^{(k)}\big\|
\lra 0
\,,
\]
where the supremum norm $\|\,.\,\|$
is taken over $\R$.
Consequently, $\gamma_\nu\ra0$
in the extrinsic strong $C^\infty$-topology
defined above.
In contrast, this sequence does not converge to $0$
in the strong $C^\infty$-topology obtained by applying
Hirsch's definition directly to the open source $(-1,0)$.
Indeed, choose $s_0\in(-2,-1)$
with $\chi(s_0)\neq0$.
The points $s_0/\nu$ form a locally finite subset
of $(-1,0)$, and choosing at these points
tolerances smaller than
$\rme^{-\nu}|\chi(s_0)|$
defines a strong $C^0$-neighbourhood of $0$
that contains none of the functions $\gamma_\nu$.
\end{exwith}


\section{Vertical integration}
\label{sec:vertint}

Let $(S,V_S,f_{\pm})$ be a shape in $(\R\times V,\rmd b+\lambda)$.


\subsection{Global flow-box theorem\label{subsec:globalflowbox}}
Consider a strict contact manifold $(M,\alpha)$
that is shaped by $(S,V_S,f_{\pm})$ via the gluing map
	$
	  \psi\co
	  \big(U,\partial M,\alpha\big)
	  \ra
	  \big(D_S,S,\rmd b+\lambda\big)
	$.
The elementary flow-box construction underlying
the global Darboux theorem \cite[Theorem~2]{eh94} yields:

\begin{prop}[Global Darboux]
 \label{prop:globalflowbox}
 The following statements are equivalent:
	\begin{enumerate}
	  \item[(a)]
	  $(M,\alpha)$ is vertically convex.
	  \item[(b)]
	  $(\hat{M},\hat{\alpha})$ has no Reeb orbits
	  trapped in backward time.
	  \item[(c)]
	  There exist a shape $(T,V_T,g_{\pm})$
	  equivalent to $(S,V_S,f_{\pm})$ with $g_-=f_-$
	  and an intrinsically defined strict contactomorphism
	  \[
	    \Phi\co
	    \big(\R\times V,D_T,\rmd b+\lambda\big)
	    \lra
	    (\hat{M},M,\hat{\alpha})
	    \,,
	  \]
	  that sends $\big(\!\End(D_T),\rmd b+\lambda\big)$
	  to $\big(\!\End(D_S),\rmd b+\lambda\big)$.
	  The restriction to
	  \[
	    Z_+:=\{b\geq g_+\}\subset\R\times V_T
	  \]
	  is of the form
	  \[
	    \Phi(b,v)=\big(b+f(v),\varphi(v)\big)
	  \]
	  for all $(b,v)\in Z_+$.
	  Here $\varphi$ is a compactly supported,
	  exact symplectomorphism of $(V_T,\rmd\lambda)$
	  with
	  \[
	  \varphi^*\lambda=\lambda-\rmd f
	  \]
	  and the function
	  \[
	  f=\varphi^*f_+-g_+
	  =(\varphi^*f_+-f_-)-(g_+-g_-)
	  \]
	  has compact support in $V_T$ and determines $g_+$
	  uniquely.
	  The restriction of $\Phi$ to $\End(D_T)\setminus Z_+$
	  is the identity map.
	  
	  In fact,
	  $\Phi(b,v)=\big(b+f(v),\varphi(v)\big)$
	  holds for all $(b,v)$ in the neighbourhood of $\End(D_T)$
	  determined by $\Phi^{-1}\big(\psi(U)\big)$,
	  where $f(v)=0$ and $\varphi(v)=v$
	  whenever $(b,v)$ lies
	  in or sufficiently close to $\End(D_T)\setminus Z_+$.
	\end{enumerate}
\end{prop}

\begin{proof}
  Because $(D_T,\rmd b+\lambda)$
  has no trapped Reeb orbits,
  the implication (c) $\Rightarrow$ (a) follows.
  The implication (a) $\Rightarrow$ (b)
  follows directly from the definition. 
  
  We are left with (b) $\Rightarrow$ (c).
  Denote by $\varphi_t$ the Reeb flow of $\hat{\alpha}$,
  which is complete by construction.
  Denote by $b_-$ the minimum of $f_-$.
  Then $\{b_-\}\times V$ is a well-defined
  hypersurface in $\hat{M}$ transverse
  to the Reeb flow $\varphi_t$.
  Using $\{b_-\}\times V$ as a Cauchy hypersurface,
  we define a map $\Phi\co\R\times V\ra\hat{M}$ by setting
  \[
    \Phi(b,v):=\varphi_{b-b_-}(b_-,v)
    \,.
  \]
  The map is injective,
  because each Reeb orbit intersects
  the separating hypersurface $\{b_-\}\times V$ at most once,
  as all intersections are transverse and
  occur with the same co-orientation.
  The map is surjective,
  because all Reeb orbits intersect $\{b_-\}\times V$
  by assumption (b).
  Therefore, $\Phi^{-1}$ exists.
  The differential $T_{(b,v)}\Phi$ evaluates on $\partial_b$ to
  the Reeb vector field $R_{\hat{\alpha}}$ of $\hat{\alpha}$
  taken at $\varphi_{b-b_-}(b_-,v)$;
  it evaluates on $T(\{b\}\times V)$ to $T_v\varphi_{b-b_-}(b_-,\,.\,)$. 
  Since $\varphi_t$ is a diffeomorphism for each $t$,
  the Reeb vector field $R_{\hat{\alpha}}$ is transverse
  to the level sets $\varphi_{b-b_-}(b_-,V)$.
  Hence, $\Phi$ is an immersion.
  The implicit function theorem yields smoothness
  of the inverse $\Phi^{-1}$.
  Moreover,
  $\Phi^*\hat{\alpha}$ evaluates on $\partial_b$ to $1$
  as we integrate $R_{\hat{\alpha}}$;
  it restricts to $\lambda$ on $T(\{b\}\times V)$ for all $b\in\R$
  because the Reeb flow of $\hat{\alpha}$ preserves $\hat{\alpha}$
  and the restriction of $\hat{\alpha}$ to $T(\{b_-\}\times V)$
  equals $\lambda$.
  Hence, $\Phi^*\hat{\alpha}=\rmd b+\lambda$.
    
  We verify the remaining properties of $\Phi$ under assumption (b).
  The hypersurface $T:=\Phi^{-1}(S)$
  in $\big(\R\times V,\rmd b+\lambda\big)$
  is transverse to $\partial_b$ along $T\setminus\Phi^{-1}(E_S)$.
  Since $\Phi$ is a diffeomorphism and $M$ is compact,
  the domain
  \[
  D_T:=\Phi^{-1}(M)
  \]
  is compact.
  By construction,
  $\Phi(b,v)=(b,v)$ for all $(b,v)\in\R\times V_S$
  with $b\leq f_-(v)$.
  Hence, for every $v\in V_S$,
  the vertical line through $v$ enters $D_T$
  through $\big(f_-(v),v\big)$
  and can leave it only through $\Phi^{-1}(S^+)$.
  Its intersection with $D_T$ is therefore a compact interval.
  By the implicit function theorem,
  its upper endpoint depends smoothly on $v$.
  Thus, $T$ is a strict vertically convex hypersurface
  with shadow set $V_T=V_S$ and lower graph $g_-=f_-$.
  The shapes $(T,V_T,g_{\pm})$ and $(S,V_S,f_{\pm})$
  are equivalent with $g_-=f_-$,
  because $T$ coincides with $S$
  along a neighbourhood of the closure of $S^-$.
  
  By construction,
  $\Phi$ is the identity map
  in a neighbourhood of $\End(D_T)\setminus Z_+$.
  On a neighbourhood of $Z_+$,
  the map $\Phi$ is a strict contactomorphism
  onto its image with respect to $\rmd b+\lambda$.
  In particular, $\Phi_*\partial_b=\partial_b$.
  Hence,
  $\Phi(b,v)=\big(b+f(v),\varphi(v)\big)$
  for all $(b,v)$ in this neighbourhood.
  Notice that $f$ has compact support in $V_T$.
  A repetition of the argument with $\Phi^{-1}|_{\End(D_S)}$
  shows that $\varphi$ is a compactly supported
  diffeomorphism of $V_T$.
  Moreover, since $\rmd b+\lambda$ equals
  $\Phi^*\big(\rmd b+\lambda\big)=\rmd b+\rmd f+\varphi^*\lambda$
  on a neighbourhood of $Z_+$,
  we get $\varphi^*\lambda=\lambda-\rmd f$ on $V_T$.
  Finally,
  for all $v\in V_T$ there exists a unique $w\in V_S$
  such that $\Phi\big(g_+(v),v\big)=\big(g_+(v)+f(v),\varphi(v)\big)$
  is equal to $\big(f_+(w),w\big)$.
  Hence,
  $w=\varphi(v)$ and $\varphi^*f_+=g_++f$.
\end{proof}

A formulation of Proposition \ref{prop:globalflowbox}
with (b) replaced by the statement
that $(\hat{M},\hat{\alpha})$ has no Reeb orbits
trapped in {\it forward} time
can be obtained by time reversal
induced by multiplying contact forms by $-1$.
Consequently,
$(M,\alpha)$ is not vertically convex
if and only if 
$(\hat{M},\hat{\alpha})$ has trapped Reeb orbits
in forward {\it and} backward time.

\begin{rem}[Free action and strict'n'exact]
\label{rem:freeactionstrnex}
  As in the final part of the proof of Proposition \ref{prop:globalflowbox}
  we have for all diffeomorphisms $\Phi$ of $\R\times V$
  the following equivalence:
  $\Phi_*\partial_b=\partial_b$ if and only if
  there exists a smooth function $f$ on $V$ and
  a smooth diffeomorphism $\varphi$ of $V$ such that
  $\Phi=(b+f,\varphi)$.
  In particular,
  $\Phi=\id$ if and only if there exists $b_0\in\R$
  such that $\Phi(b_0,v)=(b_0,v)$ for all $v\in V$.
  
  Notice that, as above,
  $\Phi$ is a strict contactomorphism of $\rmd b+\lambda$
  if and only if
  $\Phi=(b+f,\varphi)$ with $\varphi^*\lambda=\lambda-\rmd f$.
  In fact, $\Phi$ is determined uniquely
  by the exact symplectomorphism $\varphi$ of $(V,\rmd\lambda)$
  and the primitive $-f$ of $\varphi^*\lambda-\lambda$
  that takes prescribed values at the base points
  of the connected components of $V$.
The map $\varphi\mapsto f$ is continuous
in the strong $C^{\infty}$-topology described in
Section \ref{subsec:functionspaces}.
Indeed, $\varphi^*\lambda-\lambda$
depends smoothly on the first jet of $\varphi$.
On a locally finite cover by coordinate balls,
its normalised primitive $-f$ is obtained from
the standard Poincar\'e homotopy operator,
with the constants fixed at the prescribed base points.
These local constructions are continuous
with respect to all finite-order $C^r$-controls,
and local finiteness yields the asserted continuity
in the strong $C^{\infty}$-topology.
\end{rem}

\begin{rem}[Lyapunov time function]
The diffeomorphism $\Phi$ from
Proposition \ref{prop:globalflowbox}~(c)
defines a time function
$\tau:=\Phi_*b=b\circ\Phi^{-1}$
for the Reeb flow of $\hat{\alpha}$,
satisfying
$\rmd\tau(R_{\hat{\alpha}})=1$.
Its level sets form a foliation of $\hat{M}$
by $\rmd\hat{\alpha}$-symplectic hypersurfaces.
\end{rem}

\begin{rem}[Odd-symplectic uniqueness]
\label{rem:oddsymplnotnecext}
In the notation of Proposition \ref{prop:globalflowbox} (c),
the domains $(D_S,\rmd\lambda)$ and $(D_T,\rmd\lambda)$
of the equivalent shapes $(S,V_S,f_{\pm})$
and $(T,V_T,g_{\pm})$
are odd-symplectomorphic,
see \cite[Remark~4.1.2]{efz}.
Composing such an odd-symplectomorphism
with the restriction of the global flow-box contactomorphism
$\Phi$ to $D_T$ shows that
$(M,\rmd\alpha)$ is odd-symplectomorphic
to $(D_S,\rmd\lambda)$,
cf.\ Section \ref{subsec:oddsympuniness}.
\end{rem}


\subsection{Global Darboux chart}
\label{subsec:globaldarbouxchart}

Let $(M,\alpha)$ be a vertically convex strict contact manifold
shaped by $(S,V_S,f_{\pm})$.
We call the strict contactomorphism
 \[
   \Phi_{\alpha}\equiv\Phi_{(M,\alpha)}:=\Phi\co
   \big(\R\times V,D_{T_{\alpha}},\rmd b+\lambda\big)
   \lra
   (\hat{M},M,\hat{\alpha})
 \]
provided by Proposition \ref{prop:globalflowbox} (c)
the {\bf global Darboux chart}.
Since it is obtained
from the Cauchy problem for $R_{\hat{\alpha}}$,
it is intrinsically determined by $(M,\alpha)$
together with the associated gluing map $\psi$.
Here $\alpha$ varies in the space of vertically convex
contact forms on the fixed manifold $M$
that are equal to $\psi^*(\rmd b+\lambda)$
in a neighbourhood of $\partial M$.
We equip this space with the strong $C^{\infty}$-topology
induced from the affine space of smooth one-forms
with this prescribed boundary condition near $\partial M$,
see Section \ref{subsec:functionspaces}.
The Reeb vector field $R_{\hat{\alpha}}$
depends smoothly on the first jet of $\hat{\alpha}$.
Hence, standard smooth dependence of flows
on vector fields shows that $\Phi_{\alpha}$
depends continuously on $\alpha$
on every compact subset of $\R\times V$.
Thus, the global Darboux charts are compared here
in the compact-open $C^{\infty}$-topology
on the non-compact completions.

 By Proposition \ref{prop:globalflowbox} (c),
 a global Darboux chart $\Phi_{\alpha}$ defines
 	\begin{enumerate}
	 \item the {\bf monodromy} $\varphi_{\alpha}:=\varphi$,
	 \item the {\bf action difference} $f_{\alpha}:=f$ and
	 \item the {\bf super action range} $Z_{\alpha}:=Z_+$
 	\end{enumerate}
 of $(M,\alpha)$.
 The {\bf action domain} $D_{T_{\alpha}}\!$ of $(M,\alpha)$
 is determined by the equivalent shape
 $(T_{\alpha},V_{T_{\alpha}},g_{\pm}^{\alpha})$
 given by $V_{T_{\alpha}}\!=V_S$, $g_-^{\alpha}=f_-$ and
 $g_+^{\alpha}=\varphi_{\alpha}^*f_+-f_{\alpha}$.
On $Z_{\alpha}$,
the chart $\Phi_{\alpha}$ is given by
$(b+f_{\alpha},\varphi_{\alpha})$;
this expression determines all the above data.
The parameter-dependent implicit function theorem,
applied to the unique intersection of the Reeb flow
with $S^+$,
shows that $\varphi_{\alpha}$ depends continuously
on $\alpha$ with uniform control of all finite jets
on $\bar{V}_S$.
By Remark \ref{rem:freeactionstrnex},
the same holds for $f_{\alpha}$,
and hence also for
$g_+^{\alpha}=\varphi_{\alpha}^*f_+-f_{\alpha}$.
By Section \ref{subsec:functionspaces},
this is precisely continuity
in the extrinsic strong $C^{\infty}$-topology.
The corresponding dependence of
$D_{T_{\alpha}}$ and $Z_{\alpha}$
is understood through their defining graph function
$g_+^{\alpha}$.


\subsection{Monodromy space}
\label{subsec:monodromyspace}

The group
of compactly supported symplectomorphisms
\[\SS\]
of $(V_S,\rmd\lambda)$
is equipped with the strong $C^{\infty}$-topology
described in Section \ref{subsec:functionspaces}.
In view of the distinguished primitive $\lambda$
we also consider the subgroup
\[\EE\]
of exact symplectomorphisms,
equipped with the induced topology.
By a result of Giroux,
every element of $\SS$
is isotopic, through compactly supported symplectomorphisms,
to an element of $\EE$; in particular, $\EE$ meets every
path component of $\SS$,
see \cite[Lemma~7.3.4]{gei08}.

Motivated by Remark \ref{rem:freeactionstrnex}
we denote the elements of $\EE$ by $(\varphi,f)$,
where $f$ denotes the compactly supported
smooth function defined by
$\varphi^*\lambda=\lambda-\rmd f$
in case $\partial V_S$ is connected.
Otherwise,
we restrict to all $\varphi\in\SS$
for which such a $f\in C_c^{\infty}(V_S)$ exists.
Each $(\varphi,f)\in\EE$ can be identified with the
strict contactomorphism $(b+f,\varphi)$
of $\big(\R\times V,\rmd b+\lambda\big)$
with $f$ and $\varphi$ having support in $V_S$.
By Remark \ref{rem:freeactionstrnex}
the dependence of $f$ on $\varphi$ is continuous.
Accordingly, we regard $\EE$ as the graph
of this continuous assignment in
$\SS\times C_c^{\infty}(V_S)$;
this does not change its induced topology.

The {\bf monodromy space}
\[
  \MM(S)
\]
of $(S,V_S,f_{\pm})$
is the subset of all $(\varphi,f)$ in $\EE$
satisfying the {\bf shape inequality}
\[
  f_-<\varphi^*f_+-f
  \,.
\]
By Section \ref{subsec:globaldarbouxchart},
the assignment
$\alpha\mapsto(\varphi_\alpha,f_\alpha)$
defines a continuous map into $\MM(S)$
for vertically convex contact forms $\alpha$ on $M$.
Lemma \ref{lem:openshapeinequality} below
shows that $\MM(S)$ is open in $\EE$.

\begin{lem}[Openness of the shape inequality]
\label{lem:openshapeinequality}
The monodromy space $\MM(S)$ is open in $\EE$
with respect to the strong $C^{\infty}$-topology.
\end{lem}

\begin{proof}
By Section \ref{subsec:functionspaces},
the space $\EE$ is metrisable.
It is therefore enough to consider a sequence
$(\varphi_{\nu},f_{\nu})\ra(\varphi,f)$ in $\EE$
with $(\varphi,f)\in\MM(S)$
and prove that it is eventually contained in $\MM(S)$.
Set $g_{\nu,+}:=\varphi_{\nu}^*f_+-f_{\nu}$ and
$g_+:=\varphi^*f_+-f$.
Suppose, to the contrary,
that the sequence is not eventually contained in $\MM(S)$.
After passing to a subsequence,
there exist points $v_{\nu}\in V_S$ such that
$g_{\nu,+}(v_{\nu})\leq f_-(v_{\nu})$.
Since $\bar{V}_S$ is compact
and $f_+$ is continuous on $\bar{V}_S$,
strong convergence implies that
the functions $g_{\nu,+}$ converge uniformly to $g_+$
on $\bar{V}_S$.
As $g_+-f_-$ has a positive minimum
on every compact subset of $V_S$,
the points $v_{\nu}$ must approach $\partial V_S$.
After passing to a further subsequence,
we may assume that
$v_{\nu}\ra v_{\infty}$ with $v_{\infty}\in\partial V_S$.

Set $b_{\infty}:=f_+(v_{\infty})$ and
notice that $(b_{\infty},v_{\infty})\in E_S$.
At $(b_{\infty},v_{\infty})$,
the vector $\partial_b$ is tangent to $S$.
Consequently,
the image $H$ of $T_{(b_{\infty},v_{\infty})}S$
under the differential of the projection $\R\times V\ra V$
is a hyperplane in $T_{v_{\infty}}V$.
Choose local coordinates
$(s,w)\in\R\times\R^{2n-1}$ about $v_{\infty}=(0,0)$
such that $H=\{0\}\times\R^{2n-1}$.
After translating the $b$-coordinate,
we may also assume that $b_{\infty}=0$.
The implicit function theorem then allows us
to write $S$ near $(0,0,0)$
in the form $s=F(b,w)$ for a smooth function $F$.
The choice of coordinates implies $\rmd F|_{(0,0)}=0$.
After shrinking the coordinate neighbourhood
and replacing $s$ by $-s$ if necessary,
strict vertical convexity gives the local description
$V_S=\{s<h(w)\}$ for the function $h$ defined by
$h(w):=\max_bF(b,w)$,
where the maximum is taken over the local
$b$-interval under consideration.
The function $h$ is continuous
and satisfies $h(0)=0$,
although it need not be smooth.
For all sufficiently large $\nu$,
write $v_{\nu}=(s_{\nu},w_{\nu})$ with
$s_{\nu}<h(w_{\nu})=:h_{\nu}$.
Since $(\varphi_{\nu},f_{\nu})$ is compactly supported
in $V_S$, the continuous extensions satisfy
$f_-(h_{\nu},w_{\nu})=f_+(h_{\nu},w_{\nu})=
g_{\nu,+}(h_{\nu},w_{\nu})$.
Moreover, strict vertical convexity gives
$f_-(s_{\nu},w_{\nu})<f_-(h_{\nu},w_{\nu})$ 
for all sufficiently large $\nu$.
Consequently,
$g_{\nu,+}(s_{\nu},w_{\nu})\leq
f_-(s_{\nu},w_{\nu})<
g_{\nu,+}(h_{\nu},w_{\nu})$.
The mean value theorem therefore yields
$t_{\nu}\in(s_{\nu},h_{\nu})$ such that
$\partial_sg_{\nu,+}(t_{\nu},w_{\nu})>0$.
Continuity of $h$ implies that
$x_{\nu}:=(t_{\nu},w_{\nu})$ tends to 
$v_{\infty}=(0,0)$.

On the other hand,
since $(\varphi,f)$ is compactly supported in $V_S$,
we have $\varphi=\id_V$ and $f=0$
on a fixed neighbourhood of $v_{\infty}$.
Strong convergence therefore implies
$y_{\nu}:=\varphi_{\nu}(x_{\nu})\ra v_{\infty}$
and
$u_{\nu}:=T\varphi_{\nu}(\partial_s)|_{x_{\nu}}\ra\partial_s$
and also
$\rmd f_{\nu}(\partial_s)|_{x_{\nu}}\ra0$.
Since $f_+$ is continuous on $\bar{V}_S$,
the points $\big(f_+(y_{\nu}),y_{\nu}\big)$ in the exit set $S^+$
converge to $0$.
Differentiating the identity
$s=F\big(f_+(s,w),w\big)$
on $S^+$ gives
$\rmd s=\partial_bF\cdot\rmd f_++\rmd_wF$.
Along $S^+$,
$\partial_bF\ra0^-$ and $\rmd_wF\ra0$
when approaching $0$.
Evaluating the identity on $u_{\nu}$ gives
$\partial_bF\cdot\rmd f_+(u_{\nu})=
\rmd s(u_{\nu})-\rmd_wF(u_{\nu})
\ra 1$.
Consequently,
$\rmd f_+(u_{\nu})|_{y_{\nu}}$
tends to $-\infty$.
Together with
$\rmd f_{\nu}(\partial_s)|_{x_{\nu}}\ra0$,
this gives
$\partial_sg_{\nu,+}(x_{\nu})\ra-\infty$,
contradicting
$\partial_sg_{\nu,+}(x_{\nu})>0$.
Thus, the sequence is eventually contained in $\MM(S)$.
\end{proof}

\begin{rem}[Injectivity of the action map]
\label{rem:injectivityactionmap}
The action map
\[
\EE\lra C_c^{\infty}(V_S)
\,,
\quad
(\varphi,f)\longmapsto f
\,,
\]
which is continuous by Remark \ref{rem:freeactionstrnex},
is injective.
It is enough to show
that the preimage of $f=0$ consists only of $\id_{V_S}$,
i.e.\ that $\varphi^*\lambda=\lambda$
implies $\varphi=\id_{V_S}$
for every $\varphi\in\EE$.
Indeed,
such a $\varphi$ preserves the Liouville vector field $Y$
defined by $\iota_Y\rmd\lambda=\lambda$
and hence commutes with its local flow.
Since $\varphi$ is compactly supported,
there is a neighbourhood $U$ of $\partial V_S$
contained in $V_S\setminus\supp\varphi$.
Consequently,
the fixed point set of $\varphi$
contains every $Y$-trajectory that meets $U$.
The set of points
whose positive $Y$-trajectory
does not meet $U$ has empty interior.
Otherwise,
it would contain a non-empty open subset
whose forward images remain
in the compact set $V_S\setminus U$,
contradicting the exponential growth
of the $\rmd\lambda$-volume under the Liouville flow of $Y$
and the finiteness of the $\rmd\lambda$-volume of $V_S$.
Thus, the fixed point set of $\varphi$ is dense.
Since it is also closed, we conclude that $\varphi=\id_{V_S}$.
\end{rem}

\begin{rem}[Neighbourhood retract]
\label{rem:neighbourhoodretract}
Denote by $\DD(V_S)$
the compactly supported diffeomorphism group of $V_S$,
equipped with the strong $C^{\infty}$-topology
described in Section \ref{subsec:functionspaces}.
Then $\EE$ is a subspace of
\[
\FF:=\DD(V_S)\times C_c^{\infty}(V_S)
\,.
\]
For each $(\varphi,f)\in\FF$,
define the convex interpolation form by
\[
\lambda_t\equiv
\lambda_t^{\varphi,f}
:=(1-t)\lambda
+
t\big(\varphi^*\lambda+\rmd f\big)
\,,
\]
for $t\in[0,1]$,
which joins $\lambda_0=\lambda$
to $\lambda_1=\varphi^*\lambda+\rmd f$.
This defines a continuous map
$(\varphi,f,t)\mapsto\lambda_t^{\varphi,f}$
on $\FF\times[0,1]$
such that $\EE$ is given by the equation
$\dot\lambda_t=0$.
Consider the subset $\UU$ of $\FF$
defined by the requirement that
\[
\rmd\lambda_t
=(1-t)\rmd\lambda+t\varphi^*\rmd\lambda
\]
is symplectic for all $t\in[0,1]$.
By Section \ref{subsec:functionspaces},
the subset $\UU\subset\FF$ is open
since on the compact set $\bar{V}_S\times[0,1]$
non-degeneracy of $\rmd\lambda_t$
is stable under uniform $C^1$-convergence of $\varphi$.
Moreover,
$\EE\subset\UU$,
because $\lambda_t=\lambda$ on $\EE$.
We claim that $\EE$
is a strong deformation retract
of the open neighbourhood $\UU$.

Indeed, for $(\varphi,f)$ in $\UU$,
the Moser equation
\[
\iota_{X_t}\rmd\lambda_t=-\dot\lambda_t
\]
has a unique solution.
The time-dependent vector field
$X_t\equiv X_t^{\varphi,f}$
is compactly supported in $V_S$.
Denote by
$\varphi_t=\varphi_t^{\varphi,f}$
the isotopy generated by $X_t$.
The defining formula for $X_t$
and the standard theorem on parameter-dependent
ordinary differential equations,
together with the uniform control of all jets
on $\bar{V}_S$ from
Section \ref{subsec:functionspaces},
show that $X_t$ and $\varphi_t$
depend continuously on $(\varphi,f,t)$
in the strong $C^{\infty}$-topology.
Using the Lie derivative, we obtain
\[
\varphi_t^*\lambda_t=\lambda-\rmd f_t
\,,
\quad\text{with}\quad
f_t\equiv f_t^{\varphi,f}
:=-\int_0^t\varphi_s^*\bigl(\iota_{X_s}\lambda_s\bigr)\,\rmd s
\,,
\]
where
$f_t\in C_c^{\infty}(V_S)$.
Its defining formula shows that $f_t$
has the same continuous dependence.
Define a continuous map
	 \[
	 K\co\UU\times[0,1]\lra\UU
	 \,,\quad
	 \big((\varphi,f),t\big)\longmapsto
	 K_t(\varphi,f)
	 \,,
	 \]
by setting
\[
K_t(\varphi,f)
:=\Big(\varphi\circ\varphi_t,\,\varphi_t^*f+f_t\Big)
\,.
\]
First, observe that the map $K$ is well defined.
Indeed, for all $s,t\in[0,1]$,
\[
(1-t)\rmd\lambda+t\big(\varphi\circ\varphi_s\big)^*\rmd\lambda
=
\varphi_s^*\Big(
(1-t)\rmd\lambda_s+t\varphi^*\rmd\lambda
\Big)=
\varphi_s^*\rmd\lambda_{(1-s)t+s\cdot1}
\]
is symplectic since
$\rmd\lambda_{(1-s)t+s\cdot1}$
is symplectic by assumption.
Thus,
the homotopy $K$ stays inside $\UU$.
Moreover, $K_0=\id_{\UU}$ and, by the defining identity
for $f_1$, we have $K_1(\UU)\subset\EE$.
Finally,
if $(\varphi,f)\in\EE$, then $X_t=0$ identically.
It follows that $K_t|_{\EE}=\id_{\EE}$ for all $t\in[0,1]$.
Consequently,
$K$ is a strong deformation retraction
of $\UU$ onto $\EE$.

For later use we remark that
if $(\varphi,f)$ depends smoothly on a
finite-dimensional parameter,
then the parameter-dependent ordinary differential
equation theorem shows that
$K_t(\varphi,f)$ depends smoothly on this parameter and $t$.
\end{rem}


\subsection{Vertical category}
\label{subsec:verticalcategory}

We denote by
  \[
  \VV(S)
  \]
the class of all vertically convex
strict contact manifolds shaped by $(S,V_S,f_{\pm})$
modulo strict contactomorphisms
that {\bf factorise} the gluing maps
in the following sense:
Given
$\Phi\co(M',\alpha')\ra(M,\alpha)$,
the respective gluing maps $\psi'$ and $\psi$ satisfy
$\psi'=\psi\circ\Phi$
in a neighbourhood of the boundary.
With Remark \ref{rem:freeactionstrnex} it follows that
(cf.\ the arguments below)
$\Phi$ is the restriction of
$\Phi_{\alpha}\circ\Phi_{\alpha'}^{-1}$ to $M'$.
In fact:

\begin{prop}[Monodromy map]
\label{prop:monodromymap}
  The map
  \[
  \VV(S)\lra\MM(S)
  \,,
  \quad
  [M,\alpha]\longmapsto(\varphi_{\alpha},f_{\alpha})
  \,,
  \]
  is a bijection.
\end{prop}

\begin{proof}
Given another representative $(M',\alpha')$ of $[M,\alpha]$,
let $\Phi\co(M',\alpha')\ra(M,\alpha)$
be a strict contactomorphism factorising the gluing maps.
Since the gluing maps satisfy
$\id=\psi\circ\Phi\circ\psi'^{-1}$
near the boundary,
$\Phi$ extends to the completions
by the identity on $\End(D_S)$.
Denote this extension by $\hat{\Phi}$.
In the global Darboux charts,
$\hat{\Phi}$ is represented by
\[
\Phi_{\alpha}^{-1}\circ\hat{\Phi}\circ\Phi_{\alpha'},
\]
which is the identity by
Remark \ref{rem:freeactionstrnex}.
Hence,
$Z_{\alpha'}=Z_{\alpha}$,
and both global Darboux charts have the same restriction
$(b+f_{\alpha},\varphi_{\alpha})$ to this set.
In view of Section \ref{subsec:globaldarbouxchart}
the monodromy map is well defined.
We remark that
$\hat{\Phi}=\Phi_{\alpha}\circ\Phi_{\alpha'}^{-1}$
is determined uniquely.

To prove injectivity, suppose that another class
$[M',\alpha']$ has the same monodromy, that is,
$(\varphi_{\alpha'},f_{\alpha'})=(\varphi_{\alpha},f_{\alpha})$.
By Section \ref{subsec:globaldarbouxchart},
the respective super action ranges,
and hence the restrictions of the corresponding
global Darboux charts to them,
coincide.
Hence,
$(M',\alpha')$ is strictly contactomorphic to $(M,\alpha)$
via $\Phi_{\alpha}\circ\Phi_{\alpha'}^{-1}$,
which factorises the gluing maps.
  
  We define the inverse of the monodromy map
  by associating to each pair $(\varphi,f)$ in $\MM(S)$
  the class in $\VV(S)$ induced by the
  {\bf extended mapping cylinder} $M(S,\varphi)$
  constructed as follows: 
  As in Section \ref{subsec:globaldarbouxchart}
  define a shape $(T,V_T,g_{\pm})$ equivalent to $(S,V_S,f_{\pm})$
  via $g_-:=f_-$ and $g_+:=\varphi^*f_+-f$.
  Setting $\Phi:=\big(b+f,\varphi\big)$
  yields a strict contactomorphism
  of $(\R\times V,\rmd b+\lambda)$ with inverse
  $\Phi^{-1}=\big(b-\varphi_*f,\varphi^{-1}\big)$.
  We obtain a vertically convex strict contact manifold
  \[
  (D_T,\rmd b+\lambda)_{\psi}
  \]
  shaped by $(S,V_S,f_{\pm})$
  by equipping the domain $(D_T,\rmd b+\lambda)$
  with a gluing map $\psi$
  that equals $\id$ near $T^-$ and $\Phi$ near $T^+$.
  Denote by $M(S,\varphi)$ the union of
  $\R\times(V\setminus V_S)$
  with the {\bf mapping cylinder} of $(\varphi,f)$,
  i.e.\ the quotient of the subsets
  \[
  \big\{(b,v)\in\R\times V_S\mid b\leq g_+(v)\big\}
  \cup
  \big\{(b,w)\in\R\times V_S\mid b\geq f_+(w)\big\}
  \]
  by identifying
  \[
  \big(g_+(v),v\big)
  \quad\text{with}\quad
  \big(f_+(\varphi(v)),\varphi(v)\big)
  \,.
  \]
  The smooth structure on the mapping cylinder
  is induced by $\Phi^{-1}$,
  so that $M(S,\varphi)$ naturally carries a contact form
  denoted by $\alpha(S,\varphi)$.
  The resulting strict contact manifold
  $\big(M(S,\varphi),\alpha(S,\varphi)\big)$
  is the completion of $(D_T,\rmd b+\lambda)_{\psi}$.
  Observe, in view of Section \ref{subsec:globaldarbouxchart},
  that the global Darboux chart
  $\Phi_{\alpha(S,\varphi)}$ restricts to $\Phi$
  on the super action range $Z_{\alpha(S,\varphi)}$,
  i.e.\ $\varphi=\varphi_{\alpha(S,\varphi)}$ and $f=f_{\alpha(S,\varphi)}$.
\end{proof}

\begin{rem}
  Given $[M,\alpha]\in\VV(S)$,
  write $\varphi:=\varphi_{\alpha}$ for the monodromy
  and $f:=f_{\alpha}$ for the action difference.
  The proof of Proposition \ref{prop:monodromymap}
  gives a strict contactomorphism
  \[
    \Phi_{\alpha(S,\varphi)}\circ\Phi_{\alpha}^{-1}\co
    (\hat{M},\hat{\alpha})
    \lra
    \big(M(S,\varphi),\alpha(S,\varphi)\big)
    \,.
  \]
  This contactomorphism agrees with the identity
  on a neighbourhood of $\End(D_S)$
  and restricts to a strict contactomorphism
  \[
    (M,\alpha)
    \lra
    (D_T,\rmd b+\lambda)_{\psi_{\alpha}}
  \]
  onto the action domain of $(M,\alpha)$.
  This restriction factorises the gluing maps:
  $\psi_{\alpha}$ equals $\id$ near $T^-$
  and $\Phi_{\alpha}$ near $T^+$.
\end{rem}


\section{Contact-Hamiltonian diffeomorphisms}
\label{sec:conthamdiff}


\subsection{Group structure}
\label{subsec:groupstructure}

Let $(M,\alpha)$ be a strict contact manifold.
Set
\[
\xi:=\ker\alpha
\,.
\]
The following considerations are taken from
\cite[Section~2.3]{gei08}:
A diffeomorphism $\Phi$ of $M$
is called a {\bf contact-Hamiltonian diffeomorphism}
if $\Phi$ is the time-$1$ map $\Phi_1$
of an isotopy $\Phi_t$ of $M$ for which
there exists a smooth family of smooth functions
$h_t$ on $M$ such that
\[
\Phi_t^*\alpha=\rme^{h_t}\alpha
\]
for all $t\in[0,1]$.
Notice that $h_0=0$ and $\Phi_t$ preserves $\xi$.
We call $\Phi_t$ a {\bf contact-Hamiltonian isotopy}.
The associated {\bf contact-Hamiltonian} vector field
\[
X_t:=\dot{\Phi}_t\circ\Phi_t^{-1}
\]
satisfies the Lie derivative equation
\[
L_{X_t}\alpha=g_t\alpha
\,,
\qquad
g_t=(\Phi_t)_*\dot{h}_t
\,.
\]
Conversely, suppose that a smooth family of vector fields
$X_t$ satisfies $L_{X_t}\alpha=g_t\alpha$
for a smooth family of functions $g_t$.
Then the isotopy $\Phi_t$ generated by $X_t$
is contact-Hamiltonian, with
\[
h_t=\int_0^t\Phi_s^*g_s\,\rmd s
\,.
\]

The contact-Hamiltonian vector field $X_t$
defines a Hamiltonian function
via the {\bf first Hamiltonian equation}
\[
H_t:=\alpha(X_t)
\,.
\]
Contracting the Lie derivative equation with
$R_{\alpha}$ yields
\[
g_t=\rmd H_t(R_{\alpha})
\,,
\]
and hence the {\bf second Hamiltonian equation}
\[
\iota_{X_t}\rmd\alpha
=
\rmd H_t(R_{\alpha})\,\alpha-\rmd H_t
\,.
\]
Conversely, a smooth family of Hamiltonians $H_t$
uniquely determines the contact-Hamiltonian vector field
\[
X_{\!H_t}^{\alpha}:=H_tR_{\alpha}+Y_t
\,,
\]
where
\[
\alpha(Y_t)=0
\quad\text{and}\quad
\iota_{Y_t}\rmd\alpha=-\rmd H_t
\quad\text{on }\xi
\,.
\]
In view of the Lie derivative equation,
$X_{\!H_t}^{\alpha}$ is indeed contact-Hamiltonian.
The generated contact-Hamiltonian isotopy
is denoted by $\Phi_t=\Phi_t^{H_t}$.
Notice, in view of the initial definition above, that
\[
h_t=\int_0^t\Phi_s^*\big(\rmd H_s(R_{\alpha})\big)\rmd s
\,.
\]

Let $K_t$ be a second contact-Hamiltonian on $(M,\alpha)$,
and denote its vector field, isotopy, and conformal factor by
$X_{\!K_t}^{\alpha}$, $\Psi_t=\Psi_t^{K_t}$, and $k_t$,
resp.
The generating vector field of
$\Phi_t\circ\Psi_t$ is
\[
X_{\!H_t}^{\alpha}
+
(\Phi_t)_*X_{\!K_t}^{\alpha}
\,.
\]
Evaluating $\alpha$ on this vector field gives
the generating contact-Hamiltonian
\[
H_t+(\Phi_t)_*\bigl(\rme^{h_t}K_t\bigr)
\,.
\]
Since the constant identity isotopy is generated by $0$,
the inverse isotopy $\Phi_t^{-1}$ is generated by
\[
K_t=-\rme^{-h_t}\Phi_t^*H_t
\,.
\]
A contact-Hamiltonian isotopy $\Phi_t$ generated by $H_t$
is called {\bf strict} if $h_t=0$ for all $t\in[0,1]$,
or equivalently if
\[
\rmd H_t(R_{\alpha})=0
\]
for all $t\in[0,1]$.


\subsection{Contactified}
\label{subsec:contactified}

We continue the discussions from Section \ref{subsec:groupstructure}
for $(M,\alpha)$ equal to the contactisation
$(\R\times V,\rmd b+\lambda)$.
A $b$-parametric description
of the contact-Hamiltonian vector field of $H_t$
can be obtained as follows:
Observe that
$R_{\rmd b+\lambda}=\partial_b$
and decompose
$Y_t=y_t\partial_b+X_{H_t}$
w.r.t.\ the splitting $T\R\times TV$.
The first and second Hamiltonian equations then become
\[
y_t=-\lambda(X_{H_t})
\quad\text{and}\quad
\iota_{X_{H_t}}\rmd\lambda
=
(\partial_bH_t)\,\lambda-\rmd_VH_t
\,,
\]
resp., where $\rmd_V$ denotes the restriction
of the differential to $TV$.
Hence,
\[
X_{H_t}^{\rmd b+\lambda}
=
\big(H_t-\lambda(X_{H_t})\big)\partial_b+X_{H_t}
\,.
\]
Notice that $\Phi_t^{H_t}$ is strict
if and only if $H_t\equiv H_t(v)$
is independent of $b\in\R$.

Furthermore,
with Remark \ref{rem:freeactionstrnex},
$\Phi_t$ is an isotopy of strict contactomorphisms
of $(\R\times V,\rmd b+\lambda)$
if and only if
there exists a smooth family of smooth functions $F_t$ on $V$
with $F_0=0$
and a symplectic isotopy $\varphi_t$ of $(V,\rmd\lambda)$
such that
$\Phi_t=\big(b-F_t,\varphi_t\big)$ and
$\varphi_t^*\lambda=\lambda+\rmd F_t$.
Hence, in either case,
the generating vector field $X_t$ of $\Phi_t$
defines a contact-Hamiltonian function $H_t$
uniquely via the first Hamiltonian equation.
As $H_t$ must be $b$-independent,
integration with the splitting description
of $X_t=X_{H_t}^{\rmd b+\lambda}$ above yields
\[
F_t=\int_0^t
  \varphi_s^*
  \Big(
    \lambda\big(X_{H_s}\big)-H_s
  \Big)
\rmd s
\,,
\]
where $F_t(v)$, $v\in V$, is the {\bf perturbed action}
of the path $[0,t]\ni s\mapsto\varphi_s(v)$,
and
$\varphi_t$ is the Hamiltonian isotopy of $(V,\rmd\lambda)$
generated via $\dot\varphi_t=X_{H_t}|_{\varphi_t}$
by the {\bf Hamiltonian vector field} $X_{H_t}$
defined by $\iota_{X_{H_t}}\rmd\lambda=-\rmd H_t$.

Observe that
starting with a Hamiltonian isotopy
$\varphi_t$ of $(V,\rmd\lambda)$ generated by $H_t$,
taking the Lie derivative of $\lambda$
in the direction of $X_{H_t}$ yields
$\varphi_t^*\lambda=\lambda+\rmd F_t$,
where $F_t$ is the perturbed action,
see \cite[Proposition~9.3.1]{mcdsal17}.
Hence,
different choices of Hamiltonian functions $H_t$ for $\varphi_t$
result in different strict contact-Hamiltonian isotopies $\Phi_t^{H_t}$.


\subsection{Hamiltonian extension of the inverse}
\label{subsec:hamexofinv}

Let $(M,\alpha)$ be a vertically convex
strict contact manifold
shaped by $(S,V_S,f_{\pm})$
in $(\R\times V,\rmd b+\lambda)$.
Consider the global Darboux chart
\[
 \Phi_{\alpha}\co
 \big(\R\times V,D_{T_{\alpha}},\rmd b+\lambda\big)
 \lra
 (\hat{M},M,\hat{\alpha})
\]
from Proposition \ref{prop:globalflowbox} (c)
and Section \ref{subsec:globaldarbouxchart},
which equals $(b+f_{\alpha},\varphi_{\alpha})$
when restricted to a neighbourhood of $Z_{\alpha}$.

Assume that the monodromy $\varphi_{\alpha}$
of $(M,\alpha)$
is a compactly supported
Hamiltonian diffeomorphism
of $(V_S,\rmd\lambda)$.
With \cite[Definition~3.1.13]{mcdsal17}
there exist a smooth time-dependent Hamiltonian
$H_t$ on $(V,\rmd\lambda)$
and a compact set $K\subset V_S$
such that $\supp(H_t)\subset K$
for all $t\in[0,1]$.
Moreover,
$\varphi_{\alpha}=\varphi_1$ is the time-$1$ map
of the Hamiltonian isotopy $\varphi_t$
of $(V_S,\rmd\lambda)$
generated by the Hamiltonian vector field $X_{H_t}$.
By the formula for
the perturbed action $F_t$ of $\varphi_t$
in Section \ref{subsec:contactified},
$F_t$ is a smooth family of functions on $V$
with common compact support in $V_S$.
This results in a {\bf strict contact-Hamiltonian isotopy}
\[
\Phi_t^{H_t}=\big(b-F_t,\varphi_t\big)
\]
of $(\R\times V_S,\rmd b+\lambda)$,
whose time-$1$ map equals
$(b+f_{\alpha},\varphi_{\alpha})$,
so that $F_1=-f_{\alpha}$ holds ({\it sic!}).
A strict contact-Hamiltonian isotopy
$\Phi_t^{H_t}$ of $(\R\times V_S,\rmd b+\lambda)$
is called a {\bf Hamiltonian monodromy isotopy}
provided that $(\varphi_t,-F_t)\in\MM(S)$ for all $t\in[0,1]$,
where $\Phi_t^{H_t}=\big(b-F_t,\varphi_t\big)$.

\begin{lem}[Hamiltonian extension]
 \label{lem:contifnotrappedorbitshamiltonianmonod}
 Let $\Phi_t^{H_t}=\big(b-F_t,\varphi_t\big)$
 be a Hamiltonian monodromy isotopy as above.
 Let $\chi\co\R\ra[0,1]$ be a smooth cut-off function
 that is $0$ near $(-\infty,-1]$ and $1$ near $[0,\infty)$.
 
 Then these data determine
 a contact-Hamiltonian isotopy $\Psi_t$ of
 $\big(\R\times V,\xi_{\lambda}\big)$
 w.r.t.\ $\rmd b+\lambda$
 such that the following holds:
 The time-$1$ map
 \[
   \Psi_{(H_t,\chi)}:=\Psi_1\co
   \big(\R\times V,D_{T_{\alpha}},\rmd b+\lambda\big)
   \lra
   \big(\R\times V,D_S,\rme^h(\rmd b+\lambda)\big)
 \]
 is a strict contactomorphism
 such that $\Psi_{(H_t,\chi)}=\Phi_{\alpha}=(b+f_{\alpha},\varphi_{\alpha})$
 in a neighbourhood of $Z_{\alpha}$
 and $\Psi_{(H_t,\chi)}=\id$ in a neighbourhood of
 $\End(D_{T_{\alpha}})\setminus Z_{\alpha}$.
 Here $h\equiv h_{(H_t,\chi)}$
 is a uniquely determined smooth function on $\R\times V$ with support in $D_S$.
\end{lem}

\begin{proof}
Consider the inverse
$\Phi_t^{-1}=\big(b+{\varphi_t}_*F_t,\varphi_t^{-1}\big)$
of $\Phi_t=\Phi_t^{H_t}$,
which is a strict contact-Hamiltonian isotopy
generated by the Hamiltonian $-\Phi_t^*H_t$.
We define a smooth family of
shapes $(T_t,V_{T_t},g_{\pm}^t)$
in $(\R\times V,\rmd b+\lambda)$
given by $V_{T_t}=V_S$, $g_-^t=f_-$
and $g_+^t=\varphi_t^*f_++F_t$.
Observe that for each $t\in[0,1]$
the hypersurface
$T^+_t=\Phi_t^{-1}(S^+)$
is the graph of $g_+^t$
and that $g_-^t<g_+^t$ by assumption,
cf.\ Section \ref{subsec:monodromyspace}.
The path $g_+^t$ connects
$g_+^0=f_+$ with $g_+^1=g_+^{\alpha}$
and
$(T_t,V_{T_t},g_{\pm}^t)$ is a shape
equivalent to $(S,V_S,f_{\pm})$.

For all $t\in[0,1]$
and $(b,v)\in\R\times V_S$, define
\[
\tilde{\chi}_t(b,v):=
\chi\left(\frac{b-g_+^t(v)}{g_+^t(v)-f_-(v)}\right)
\,.
\]
As the support of $-\Phi_t^*H_t$
is contained in $V_S$ for all $t\in[0,1]$,
we obtain a smooth family
of contact-Hamiltonians $K_t:=-\tilde{\chi}_t\Phi_t^*H_t$
on $\big(\R\times V,\rmd b+\lambda\big)$.
Denote by $\tilde{\Psi}_t$ the generated
contact-Hamiltonian isotopy,
so that $\Psi_t:=\tilde{\Psi}_t^{-1}$
is the contact-Hamiltonian isotopy generated by
\[
 \tilde{H}_t:=
 -\rme^{-\int_0^t\tilde{\Psi}_s^*(K_s)_b\,\rmd s}\;
 \tilde{\Psi}_t^*K_t
 \,.
\]
Abbreviating $Z_t:=\{b\geq g_+^t\}$
we get that $\tilde{H}_t$ equals $H_t$ in a neighbourhood of $Z_t$
and $0$ in a neighbourhood of $\End(D_{T_t})\setminus Z_t$.
Hence,
$\Psi_t=\Phi_t$ in a neighbourhood of $Z_t$
and $\Psi_t=\id$ in a neighbourhood of $\End(D_{T_t})\setminus Z_t$.
On $D_{T_t}$ we have
\[
\Psi_t^*(\rmd b+\lambda)=
\rme^{\int_0^t\Psi_s^*(\tilde{H}_s)_b\,\rmd s}
(\rmd b+\lambda)
\,.
\]
Setting
\[
 h_t:=
 -{\Psi_t}_*
 \int_0^t\Psi_s^*(\tilde{H}_s)_b\,\rmd s
 \,,
\]
we get a smooth family of compactly supported
functions $h_t$ on $\Int(D_S)$ such that
\[
 \Psi_t\co
 \big(\R\times V,D_{T_t},\rmd b+\lambda\big)
 \lra
 \big(\R\times V,D_S,\rme^{h_t}(\rmd b+\lambda)\big)
 \,.
\]
The claim follows by setting $\Psi_{(H_t,\chi)}:=\Psi_1$ and
$h_{(H_t,\chi)}:=h_1$.
\end{proof}

Observe that the contact-Hamiltonian isotopy $\Psi_t$
depends continuously on the path
$\Phi_t^{H_t}=\big(b-F_t,\varphi_t\big)$ in $\MM(S)$
and is, by Remark \ref{rem:freeactionstrnex},
the global Darboux chart
$\Psi_t=\Phi_{\rme^{h_t}(\rmd b+\lambda)}$ 
of $\big(D_S,\rme^{h_t}(\rmd b+\lambda)\big)$
for each $t\in[0,1]$.
In particular, at time $t=1$:
\[
\Psi_{(H_t,\chi)}=\Phi_{\rme^h(\rmd b+\lambda)}
\,,
\]
cf.\ Section \ref{subsec:utransform} below.
Replacing the generating Hamiltonian $H_t$
by $K_t$ inside $\MM(S)$ yields
an analogous contactomorphism
$\Psi_{(K_t,\chi)}=\Phi_{\rme^k(\rmd b+\lambda)}$
with $k\equiv k_{(K_t,\chi)}$ as in Lemma
\ref{lem:contifnotrappedorbitshamiltonianmonod}.
The two contactomorphisms
differ by the contactomorphism
\[
  \Psi_{(H_t,\chi)}\circ(\Psi_{(K_t,\chi)})^{-1}\co
  \big(\R\times V,D_S,\rme^k(\rmd b+\lambda)\big)
  \lra
  \big(\R\times V,D_S,\rme^h(\rmd b+\lambda)\big)
  \,.
\]
The action domains $D_{T_t}$ corresponding to $H_t$ and $K_t$
of the respective contact-Hamiltonian isotopies $\Psi_t$
used in the proof of Lemma
\ref{lem:contifnotrappedorbitshamiltonianmonod}
coincide for all $t\in[0,1]$ if and only if
$H_t$ and $K_t$ define the same path in $\MM(S)$,
cf.\ Section \ref{subsec:verticalcategory}.

\begin{rem}[Contactomorphic]
\label{rem:contifnotrappedorbitshamiltonianmonod}
Under the assumptions of Lemma
\ref{lem:contifnotrappedorbitshamiltonianmonod}
the composition
\[
 \Phi_{\rme^h(\rmd b+\lambda)}\circ\Phi_{\alpha}^{-1}\co
 (\hat{M},M,\hat{\alpha})
 \lra
 \big(\R\times V,D_S,\rme^h(\rmd b+\lambda)\big)
\]
is a strict contactomorphism factorising the gluing maps,
cf.\ Section \ref{subsec:verticalcategory}.
Indeed, the composition is the identity in a neighbourhood of
$\big(\!\End(D_S),\rmd b+\lambda\big)$ and
defines a strict contactomorphism
$(M,\alpha)\ra\big(D_S,\rme^h(\rmd b+\lambda)\big)$ that
restricts to the gluing map
$\psi$ of $(M,\alpha)$ in a neighbourhood of $\partial M$.
Hence,
$[M,\alpha]=\big[D_S,\rme^h(\rmd b+\lambda)\big]$
in $\VV(S)$ provided that the monodromy of $[M,\alpha]$
fits into a Hamiltonian monodromy isotopy
as ensured by Lemma \ref{lem:pi1ofrelemistrivial} below,
see Section \ref{subsec:utransform}.
\end{rem}


\subsection{Changing the action}
\label{subsec:changingtheaction}

Consider a contact-Hamiltonian isotopy $\Phi_t=\Phi_t^{H_t}$
on a strict contact manifold $(M,\alpha)$.
Let $\tau\co[0,1]\ra[0,1]$ be a smooth map
constant near the boundary
with $\tau(0)=0$ and $\tau(1)=1$.
The substitution $\Psi_t=\Phi_{\tau(t)}$ turns $\Phi_t$ into a
{\bf technical isotopy}.
Observe that $\Phi_{(1-s)t+s\tau(t)}$, $s\in[0,1]$,
is a homotopy relative to the endpoints
between the original and the technical isotopy.
Taking the time derivative yields that
$\Psi_t$ is generated by $\tau'(t)X^{\alpha}_{H_{\tau(t)}}$,
so that the generating Hamiltonian is equal to $\tau'(t)H_{\tau(t)}$.
Observe that $\Phi_{\tau(t)}^*\alpha=\rme^{h_{\tau(t)}}\alpha$
for all $t\in[0,1]$.
In particular, $\Psi_t$ is strict if and only if $\Phi_t$ is.
If $\sigma$ is another such map and
$\tau'$ and $\sigma'$ have disjoint supports, then
$\Phi_{\tau(t)}\circ\Phi_{\sigma(t)}^{-1}$
is a contractible loop of contact-Hamiltonian diffeomorphisms.
Assume, in addition,  that
$\Phi_t=(b-F_t,\varphi_t)$
is strict on $(\R\times V,\rmd b+\lambda)$,
so that $H_t$ is $b$-independent.
With
$\Phi_{\tau(t)}=\big(b-F_{\tau(t)},\,\varphi_{\tau(t)}\big)$,
so that
\[
\Phi_{\sigma(t)}^{-1}=
\Big(
b+(\varphi_{\sigma(t)})_*F_{\sigma(t)},\,
\varphi_{\sigma(t)}^{-1}
\Big)
\,,
\]
the loop reads as
\[
\Phi_{\tau(t)}\circ\Phi_{\sigma(t)}^{-1}=
\Big(
  b-(\varphi_{\sigma(t)})_*
    \big(
      F_{\tau(t)}-F_{\sigma(t)}
    \big),\,
  \varphi_{\tau(t)}\circ\varphi_{\sigma(t)}^{-1}
\Big)
\,.
\]

To simplify the action coordinate,
assume that the closure of the shadow set $(V_S,\lambda)$
of a shape $(S,V_S,f_{\pm})$ is a Liouville domain
and that $f_{\pm}|_{\partial V_S}=0$.
We will call such a shape a {\bf Liouville shape}.
The smooth boundary $\partial V_S$ is a strict contact manifold
equipped with the contact form $\alpha:=\lambda|_{T\partial V_S}$.
Choose a symplectisation collar neighbourhood
$\big((-a_S,0]\times\partial V_S,\rme^a\alpha\big)$
for some positive $a_S\in\R$
such that
\[
f_-<0<f_+
\quad\text{holds near}\quad
[-a_S,0)\times\partial V_S
\,.
\]
Furthermore,
in order to define a {\bf potential well},
choose a smooth function
\[
\chi_V\co V\ra[-1,0]
\]
that equals $-1$ on the complement of the collar in $V_S$
and $0$ on the complement of $V_S$.
On the collar, assume that
$\chi_V\equiv\chi_V(a)$ is monotonically increasing,
identically equal to $-1$ near $[-a_S,-a_0]$,
and identically equal to $0$ near $[-a_0/2,0]$.
The autonomous Hamiltonian defining the potential well
is $P:=P_0\cdot\chi_V$ for $P_0\geq0$.
The corresponding Hamiltonian vector field equals
$X_P=\rme^{-a}P'(a)R_{\alpha}$ on the collar neighbourhood
and equals $0$ otherwise.
As $\rme^a\alpha(X_P)=P'(a)$ and
$X_P$ preserves $P$ and $P'$,
the action is equal to $F_t=t(P'-P)$
and, hence, non-negative.
Therefore, we get
\[
\Lambda_{\tau(t),\sigma(t)}^{P_0,a_0}:=
\Phi_{\tau(t)}\circ\Phi_{\sigma(t)}^{-1}=
\Big(
  b-\big(\tau(t)-\sigma(t)\big)(P'-P),\,
  \varphi_{\tau(t)-\sigma(t)}^{X_P}
\Big)
\]
for the {\bf potential well loop}.

\begin{lem}[Hamiltonian monodromy isotopy]
\label{lem:pi1ofrelemistrivial}
Let $(S,V_S,f_{\pm})$ be a Liouville shape
in $(\R\times V,\rmd b+\lambda)$.
Consider a strict contact-Hamiltonian isotopy $\Phi_t=\Phi_t^{H_t}$
of the form $\Phi_t=(b-F_t,\varphi_t)$ in $\EE$
with endpoints of the path $\Phi_t$ in $\MM(S)$.

Let $a_0\equiv a_0(H_t)>0$ be so small
that the union of the supports of $H_t$ is disjoint from
the symplectic collar
$\big([-a_0,0]\times\partial V_S,\rme^a\alpha\big)$.
Define
\[
\Delta(H_t)_t:=f_--\varphi_t^*f_+-F_t
\]
and set
\[
P_0\equiv P_0(H_t):=
\max\big(\Delta(H_t)_t\big)+
\max\big(|H_t|\big)
\,,
\]
where both maxima are taken over all $t\in[0,1]$
and the closure of $V_S$.
Choose a smooth function $\tau\co[0,1/4]\ra[0,1]$ 
constant near the boundary
with $\tau(0)=0$ and $\tau(1/4)=1$.
Extend $\tau$ constantly to $\R$ and define
$\varrho(t):=\tau(t-1/4)$ and $\sigma(t):=\tau(t-1/2)$.

Then
\[
\Phi_{\varrho(t)}\circ\Lambda_{\tau(t),\sigma(t)}^{P_0,a_0}
\]
is a path in $\MM(S)$ homotopic to $\Phi_t$
relative to the endpoints
and is constantly equal to $\id$
for the trivial path $\Phi_t=\id$.
\end{lem}

\begin{proof}
  Denote by $\Psi_t=(b-G_t,\psi_t)$ the isotopy in question.
  The aim is to verify the shape inequality
  $f_-<\psi_t^*f_++G_t$ for all $t\in[0,1]$.
  Since $X_P$ is supported in $[-a_0,0]\times\partial V_S$,
  where $\varphi_t=\id$ and $F_t=0$ for all $t\in[0,1]$,
  we obtain
  \[
    \Psi_t=
    \Big(
      b-\big(\tau(t)-\sigma(t)\big)(P'-P)-F_{\varrho(t)},\,
      \varphi_{\varrho(t)}\circ\varphi_{\tau(t)-\sigma(t)}^{X_P}
    \Big)
    \,.
  \]
  Outside this collar, this reads as
  \[
  \Big(b-\big(\tau(t)-\sigma(t)\big)P_0-F_{\varrho(t)},\,
  \varphi_{\varrho(t)}\Big)
  \]
  and on the collar as
  \[
  \Big(b-\big(\tau(t)-\sigma(t)\big)(P'-P),\,
  \varphi_{\tau(t)-\sigma(t)}^{X_P}\Big)
  \,.
  \]
  Therefore,
  we have to verify the shape inequality in the two forms
  \[
  f_-<
  \varphi_{\varrho(t)}^*f_++F_{\varrho(t)}
  +\big(\tau(t)-\sigma(t)\big)P_0
  \]
  and
  \[
  f_-<
  \big(\varphi_{\tau(t)-\sigma(t)}^{X_P}\big)^*f_+
  +\big(\tau(t)-\sigma(t)\big)(P'-P)
  \]
  on the respective domains.
  
  The right-hand side of the first variant takes,
  on the consecutive time intervals
  $[0,1/4]$, $[1/4,1/2]$ and $[1/2,1]$,
  the form
  \[
  f_++\tau(t)P_0
  \,,
  \quad
  \varphi_{\varrho(t)}^*f_++F_{\varrho(t)}+P_0
  \quad\text{and}\quad
  \varphi_1^*f_++F_1
  +\big(1-\sigma(t)\big)P_0
  \,.
  \]
  For every $t\in[0,1]$,
  the function $\Delta(H_t)_t$ vanishes along $\partial V_S$,
  because $f_-=0=f_+$ there and
  $(\varphi_t,-F_t)$ is compactly supported in $V_S$.
  Hence,
  the maximum occurring in the definition of $P_0$
  is non-negative, i.e.\ $P_0(H_t)\geq0$.
  Furthermore,
  for all $t\in[0,1]$,
  we have $P_0(H_t)>\Delta(H_t)_t$ 
  provided $H_t$ is not identically $0$;
  otherwise $P_0(0)=0$,
  in which case the shape inequality for $\Phi_t=\id$
  is satisfied trivially.
  Hence,
  the above right-hand sides are estimated
  from below as
  \[
  \geq f_+>f_-
  \,,
  \quad
  >\varphi_{\varrho(t)}^*f_++F_{\varrho(t)}+\Delta(H_t)_{\varrho(t)}
  =f_-
  \quad\text{and}\quad
  \geq\varphi_1^*f_++F_1>f_-
  \,,
  \]
  resp., as claimed.
  
  Because $\big(\tau(t)-\sigma(t)\big)(P'-P)\geq0$,
  the second variant of the shape inequality
  can be reduced to
  \[
  f_-<
  \big(\varphi_{\tau(t)-\sigma(t)}^{X_P}\big)^*f_+
  \,.
  \]
  By the choice of $a_0$ we have
  $f_-<0<f_+$ on $(-a_0,0)\times\partial V_S$
  and the claim follows.
\end{proof}

\begin{proof}[{\bf Proof of Theorem \ref{thm:monodromy}}]
Lemma \ref{lem:pi1ofrelemistrivial} provides
the existence of a Hamiltonian monodromy isotopy
required in Remark \ref{rem:contifnotrappedorbitshamiltonianmonod}.
\end{proof}


\subsection{Alexander dilation\label{subsec:alexdilation}}
The following construction is inspired by
\cite[Section 2.6.2]{gei08}.
Let $(S,V_S,f_{\pm})$ be a Liouville shape
in $(\R\times V,\rmd b+\lambda)$.
The Liouville vector field $Y$ of $(V,\lambda)$
is defined via $\iota_Y\rmd\lambda=\lambda$.
There exists $\varepsilon_S>0$ such that
$(-\infty,\varepsilon_S]\times V_S$
is contained in the domain of the Liouville flow
$(t,v)\mapsto\varphi_t^Y(v)$.
Because
\[
(\varphi^Y_t)^*\lambda=\rme^t\lambda
\,,
\]
we get
$(\varphi^Y_t)^*\rmd\lambda=\rme^t\rmd\lambda$
whenever defined.
Hence, for all $\varphi\in\SS$
and $0\leq t<\varepsilon_S$,
the {\bf symplectic Alexander dilation}
\[
a_t(\varphi):=
\varphi^Y_{-t}\circ\varphi\circ\varphi^Y_t
\]
belongs to $\SS$.
Indeed, it is symplectic and
\[
\supp\big(a_t(\varphi)\big)
=
\varphi^Y_{-t}\big(\supp(\varphi)\big)
\,.
\]
Consider $\Phi=(b+f,\varphi)$ in $\EE$
and observe that
$\big(a_t(\varphi)\big)^*\lambda=\lambda-\rmd\beta_t(f)$,
where
\[
\beta_t(f):=\rme^{-t}f\circ\varphi^Y_t
\,.
\]
Since
\[
\supp\big(\beta_t(f)\big)=
\varphi^Y_{-t}\big(\supp(f)\big)
\,,
\]
the {\bf contact Alexander dilation}
\[
A_t(\Phi):=
\Big(b+\beta_t(f),a_t(\varphi)\Big)
\]
belongs to $\EE$
for all $0\leq t<\varepsilon_S$.
Notice that $A_t(\id)=\id$.

\begin{prop}[Homotopy equivalence]
\label{prop:homoequiv}
Let $(S,V_S,f_{\pm})$ be a Liouville shape
in $(\R\times V,\rmd b+\lambda)$.
Denote by $\EE_0$
the path component of $(\id_{V_S},0)$ in $\EE$
and set $\MM_0:=\MM(S)\cap\EE_0$.
Then the inclusion
\[
  \MM_0\lra\EE_0
\]
is a homotopy equivalence.
\end{prop}

\begin{proof}
It is enough to show that $\MM_0\hookrightarrow\EE_0$
is a weak homotopy equivalence.
By Section \ref{subsec:functionspaces},
the strong $C^{\infty}$-topologies on
$C_c^{\infty}(V_S)$ and
$C_c^{\infty}(V_S,TV_S)$
are induced by their respective countable families
of uniform $C^k$-seminorms on $\bar{V}_S$.
Thus, both spaces are metrisable locally convex
topological vector spaces.
With this topology,
a local addition on the ambient manifold $V$
gives $\DD(V_S)$ the structure
of a metrisable manifold modelled on
$C_c^{\infty}(V_S,TV_S)$.
Consequently, the space
\[
\FF=\DD(V_S)\times C_c^{\infty}(V_S)
\,,
\]
introduced in Remark \ref{rem:neighbourhoodretract},
is a metrisable manifold modelled on a metrisable
locally convex topological vector space.

The Dugundji extension theorem
\cite[Theorem~4.1]{du51}
states that every continuous map from a closed subset
of a metric space into a locally convex topological vector space
extends to the whole space.
Applying this result to the identity map on any closed copy
of the target space shows that every metrisable locally convex
topological vector space is an absolute retract
and hence, in particular, an absolute neighbourhood retract.
Together with Hanner's local characterisation
of the ANR property \cite[Theorem~3.3]{ha51},
this implies that $\FF$ is an ANR.
Hence, the open subset $\UU\subset\FF$ is also an ANR.
By Remark \ref{rem:neighbourhoodretract},
the space $\EE$ is a retract of $\UU$
and therefore is itself an ANR.

Since metrisable ANRs are locally path-connected,
the path component $\EE_0$ is open in $\EE$
and is itself an ANR.
By Lemma \ref{lem:openshapeinequality},
$\MM_0$ is open in $\EE_0$
and therefore is itself an ANR.
Every metrisable ANR has the homotopy type
of a CW complex \cite[Theorem~2]{mi59}.
Consequently, both $\MM_0$ and $\EE_0$
have the homotopy type of CW complexes.
Whitehead's theorem
\cite[Theorem~4.5; see also p.~352]{hat02}
now implies that a weak homotopy equivalence
$\MM_0\hookrightarrow\EE_0$
upgrades to a homotopy equivalence.

We now prove the required weak homotopy equivalence.
Represent an element of $\pi_i(\EE_0,\MM_0)$
by a continuous map
\[
\Gamma\co
(D^i,S^{i-1})\lra(\EE_0,\MM_0)
\,,
\quad
\Gamma(0)=(\id_{V_S},0)
\,.
\]
We first arrange uniform compact support.
Let $A_s$ be the contact Alexander dilation
introduced above.
Since $\Gamma(S^{i-1})$ is compact in $\MM(S)$,
openness of $\MM(S)$ and continuity of $A_s$
give a number $a\in(0,\varepsilon_S)$ such that
\[
A_s\big(\Gamma(S^{i-1})\big)\subset\MM(S)
\]
for all $s\in[0,a]$.
Thus, $A_s\circ\Gamma$ is a homotopy
of relative representatives.
Replacing $\Gamma$ by $A_a\circ\Gamma$,
we may therefore assume that all members
of the family are supported
in the interior of the fixed compact set
\[
V_a:=
\varphi^Y_{-a}\big(\bar{V}_S\big)
\,.
\]

We next replace $\Gamma$ by a representative
that is smooth in the finite-dimensional parameter.
The preceding local-addition argument,
applied with $\Int(V_a)$ in place of $V_S$,
shows that the subspace of $\FF$
consisting of pairs supported in $\Int(V_a)$
is a metrisable locally convex manifold.
Its intersection with the open neighbourhood
$\UU\subset\FF$
from Remark \ref{rem:neighbourhoodretract}
is open.
Regard $\Gamma$ as a map into this intersection.
By relative smooth approximation for maps
from finite-dimensional manifolds with corners
into locally convex manifolds,
cf.\ \cite[Theorem~11 and Proposition~13]{woc09},
the map $\Gamma$ can be approximated
within this intersection
by a smooth map $\Gamma'\co D^i\ra\UU$
through a homotopy fixing $0\in D^i$,
such that all members of the homotopy
are supported in $\Int(V_a)$.

Recall the strong deformation retract $K$ from
Remark \ref{rem:neighbourhoodretract}.
Since $K_1\co\UU\ra\EE$ is continuous,
$K_1|_{\EE}=\id_{\EE}$,
and $\MM(S)$ is open in $\EE$,
the approximation can be chosen such that,
after applying $K_1$,
the homotopy maps $S^{i-1}$ into $\MM(S)$.
Every map in the resulting homotopy
takes $0$ to $(\id_{V_S},0)$
and, since $D^i$ is connected,
takes values in $\EE_0$.
On $S^{i-1}$,
the resulting homotopy remains in $\MM(S)$
and starts at $\Gamma|_{S^{i-1}}$.
Its boundary values therefore remain in $\MM_0$.
Thus,
$K_1\circ\Gamma'$
represents the same element of
$\pi_i(\EE_0,\MM_0)$.
The parameter-dependent Moser construction
shows that $K_1\circ\Gamma'$ is smooth
in the parameter on $D^i$.
Moreover, the Moser retraction preserves support
in $\Int(V_a)$,
since its Moser vector field vanishes
outside $\Int(V_a)$.
For ease of notation,
we continue to denote $K_1\circ\Gamma'$ by $\Gamma$.

Write points of $D^i$ in the form $ru$
with $r\in[0,1]$ and $u\in S^{i-1}$,
and write
\[
\Gamma(ru)=(\varphi_{r,u},f_{r,u})
\,.
\]
For every $u$,
the path $r\mapsto\varphi_{r,u}$
is a smooth path of compactly supported
exact symplectomorphisms.
It is therefore a Hamiltonian isotopy by
\cite[Proposition~9.3.1 and
Corollary~9.3.3, pp.~369--370]{mcdsal17}.
The formula in the proof of
\cite[Proposition~9.3.1]{mcdsal17}
shows, moreover, that its generating Hamiltonian
$H_{r,u}$ can be chosen smoothly in $(r,u)$
and with support in the same fixed compact subset.
Setting $F_{r,u}:=-f_{r,u}$ gives
the corresponding strict contact-Hamiltonian isotopy
\[
\Phi_{r,u}^{H_{r,u}}=
\big(b-F_{r,u},\varphi_{r,u}\big)
\]
with
\[
\Phi_{0,u}^{H_{0,u}}=\id
\quad\text{and}\quad
\Phi_{1,u}^{H_{1,u}}\in\MM_0
\,.
\]
As $S^{i-1}$ is compact
and all Hamiltonians have support
in the fixed compact set $V_a$,
the number $a_0$ in
Lemma \ref{lem:pi1ofrelemistrivial}
can be chosen independently of $u$.
For every $u\in S^{i-1}$, set
$P_0(u):=P_0(H_{\cdot,u})$
as in that lemma.
Since the maximum is Lipschitz continuous
with respect to the supremum norm,
the {\bf breathing lid} $u\mapsto P_0(u)$ is continuous.
Consequently, by their explicit construction,
the potential well loops and the homotopies in
Lemma \ref{lem:pi1ofrelemistrivial}
depend continuously on $u$.
Using the functions $\tau$, $\sigma$ and $\varrho$
from that lemma, define
\[
\Gamma''(ru):=
\Phi_{\varrho(r),u}^{H_{\varrho(r),u}}
\circ
\Lambda_{\tau(r),\sigma(r)}^{P_0(u),a_0}
\,.
\]
The cut-off functions are constant near their endpoints.
Hence, $\Gamma''(ru)=\id$
for all sufficiently small $r$,
independently of $u$, and
$\Gamma''(u)=\Gamma(u)$
for $u\in S^{i-1}$.
Thus, $\Gamma''$ defines a continuous map
from $D^i$ into $\MM(S)$.
Since $D^i$ is connected and
$\Gamma''(0)=\id$,
its image is contained in $\MM_0$.
Moreover, the homotopies provided by
Lemma \ref{lem:pi1ofrelemistrivial}
assemble to a homotopy of relative representatives
from $\Gamma$ to $\Gamma''$.
Since $\Gamma''$ takes all of $D^i$
into $\MM_0$,
the relative homotopy class represented by $\Gamma$
is trivial.

Therefore,
$\pi_i(\EE_0,\MM_0)$ is trivial for all $i\geq1$.
Since $\EE_0$ is path-connected,
the long exact homotopy sequence
of the pair $(\EE_0,\MM_0)$,
together with triviality of $\pi_1(\EE_0,\MM_0)$,
implies that $\MM_0$ is path-connected.
Hence, the inclusion of $\MM_0$ into $\EE_0$ 
is a weak homotopy equivalence.
By the discussion at the beginning of the proof,
it is a homotopy equivalence.
\end{proof}

We call the Liouville shape $(S,V_S,f_{\pm})$
in $(\R\times V,\rmd b+\lambda)$
{\bf sub- and super-homogeneous}
provided that
\[
\pm\beta_t(f_{\pm})
\leq\pm f_{\pm}
\]
on $\varphi^Y_{-t}(V_S)$ for all $0\leq t<\varepsilon_S$.
This is satisfied, for example,
for shape functions $f_{\pm}$ with $f_-<0<f_+$ on $V_S$,
for which
$(-\infty,\varepsilon_S)\ni s\mapsto\pm f_{\pm}\big(\varphi^Y_s(w)\big)$
is (radially) decreasing for all $w\in\partial V_S$.

\begin{lem}
\label{lem:alexshapeineq}
Let $(S,V_S,f_{\pm})$ be a sub- and super-homogeneous
Liouville shape in $(\R\times V,\rmd b+\lambda)$.
Then $A_t(\Phi)\in\MM(S)$
for all $\Phi\in\MM(S)$ and $0\leq t<\varepsilon_S$.
\end{lem}

\begin{proof}
Write $\Phi=(b+f,\varphi)$ and assume that
the shape inequality $f_-<\varphi^*f_+-f$
holds.
Outside $\varphi^Y_{-t}(V_S)$,
the contact Alexander dilation $A_t(\Phi)$ is the identity,
so that the claim follows from $f_-<f_+$.
On $\varphi^Y_{-t}(V_S)$,
observe that
\[
\big(a_t(\varphi)\big)^*f_+-\beta_t(f)>
\big(a_t(\varphi)\big)^*f_++\beta_t(f_-)-\beta_t(\varphi^*f_+)
\]
by the shape inequality for $(\varphi,f)$.
An application of sub-homogeneity yields
\[
\big(a_t(\varphi)\big)^*f_+\geq
\big(a_t(\varphi)\big)^*\beta_t(f_+)=
\beta_t(\varphi^*f_+)
\,.
\]
Together with super-homogeneity $\beta_t(f_-)\geq f_-$,
the shape inequality follows for $\big(a_t(\varphi),\beta_t(f)\big)$.
\end{proof}


\subsection{Hamiltonian contraction spaces\label{subsec:hamcontrspace}}

Let $(S,V_S,f_{\pm})$ be a shape in
the contactisation $(\R\times V,\rmd b+\lambda)$.
We call $(V_S,\rmd\lambda)$
a {\bf Hamiltonian contraction space}
if the following conditions are satisfied:
 	\begin{enumerate}
	 \item The boundary $\partial V_S$ is connected.
	 \item There exists a {\bf time}-$C^1$
	 {\bf contraction} of $\SS$, i.e.\ a continuous map
	 \[
	 C\co\SS\times[0,1]\lra\SS
	 \,,\quad
	 (\varphi,t)\longmapsto
	 C_t(\varphi):=C(\varphi,t)
	 \,,
	 \]
	 satisfying $C_0=\id_{\SS}$, $C_1=\id_{V_S}$ and
	 $C_t(\id_{V_S})=\id_{V_S}$ for all $t\in[0,1]$,
	 such that for every $\varphi\in\SS$ the evaluation map
	 \[
	 V_S\times[0,1]\lra V_S
	 \,,\quad
	 (v,t)\longmapsto C_t(\varphi)(v)
	 \,,
	 \]
	 is smooth and
	 \[
	 \SS\times[0,1]\lra C_c^{\infty}(V_S,TV_S)
	 \,,\quad
	 (\varphi,t)\longmapsto
	 \frac{\partial}{\partial t}C(\varphi,t)
	 \,,
	 \]
	 is continuous in the strong $C^{\infty}$-topology.
	 \item The group $\SS$ consists entirely
	 of Hamiltonian diffeomorphisms.
 	\end{enumerate}

Using the group structure,
$C_t$ can be substituted by
$\big(C_t(\id_{V_S})\big)^{-1}\circ C_t$
in order to achieve
$C_t(\id_{V_S})=\id_{V_S}$ for all $t\in[0,1]$
if necessary.
Notice that the time-$C^1$ condition requires only
the first time derivative to depend continuously
on $(\varphi,t)$;
the individual contraction paths are required to be smooth.

The construction in \cite[p.~417]{agz22}
applies to a Hamiltonian contraction space as follows:
Observe that
$\varphi_t:=C_{1-t}(\varphi)$ is a compactly supported
Hamiltonian isotopy of $(V_S,\rmd\lambda)$
with generating vector field
$X_t(\varphi):=\dot\varphi_t\circ\varphi_t^{-1}$.
By the time-$C^1$ assumption, the map
	 \[
	 X\co\SS\times[0,1]\lra C_c^{\infty}(V_S,TV_S)
	 \,,\quad
	 (\varphi,t)\longmapsto
	 X_t(\varphi):=X(\varphi,t)
	 \,,
	 \]
is continuous in the strong $C^{\infty}$-topology,
and the evaluation
$(v,t)\mapsto X_t(\varphi)(v)$ is smooth
for every fixed $\varphi\in\SS$.
Since the path $\varphi_t$ consists of
Hamiltonian diffeomorphisms,
the compactly supported $1$-form
$\iota_{X_t(\varphi)}\rmd\lambda$
is exact.
By the continuous primitive construction from
Remark \ref{rem:freeactionstrnex},
we obtain a continuous map
	\[
	H\co\SS\times[0,1]\lra C_c^{\infty}(V_S)
	\,,\quad
	(\varphi,t)\longmapsto
	H_t(\varphi):=H(\varphi,t)
	\,,
	\]
such that
$\iota_{X_t(\varphi)}\rmd\lambda=-\rmd H_t(\varphi)$.
For every fixed $\varphi$, the map
$(v,t)\mapsto H_t(\varphi)(v)$ is smooth,
so that $\varphi_t$ is the Hamiltonian isotopy
generated by the Hamiltonian $H_t(\varphi)$.
Observe that
$H_t(\id_{V_S})=0$ for all $t\in[0,1]$.

\begin{prop}
\label{prop:mscontractible}
Let $(S,V_S,f_{\pm})$ be a sub- and super-homogeneous
Liouville shape in $(\R\times V,\rmd b+\lambda)$
and assume that $(V_S,\rmd\lambda)$
is a Hamiltonian contraction space.
Then $\MM(S)\simeq\{(\id_{V_S},0)\}$
via a time-$C^1$ strong deformation retraction.
\end{prop}

\begin{proof}
Notice that $\SS=\EE$,
see Section \ref{subsec:contactified}.
Hence,
taking a time-$C^1$ contraction $C$
as in the above definition,
the path $\varphi_t=C_{1-t}(\varphi)$
induces a strict contact-Hamiltonian isotopy
$\Phi_t=(b-F_t,\varphi_t)$
generated, as above, by the Hamiltonian function
$H_t=H_t(\varphi)$.
The perturbed action function
$F_t=F_t\big(H_t(\varphi)\big)$
can be computed as in Section \ref{subsec:contactified}.
Therefore,
$\Phi_t$ is smooth in $t$ and
depends continuously on $\varphi$.

Consider the restrictions
$C\co\MM(S)\times[0,1]\ra\SS$
and
$H\co\MM(S)\times[0,1]\ra C_c^{\infty}(V_S)$,
so that the corresponding path
$(\varphi_t,-F_t)$ lies in $\EE$
and has endpoints in $\MM(S)$.
In order to turn $\Phi_t$ into a
Hamiltonian monodromy isotopy,
we consider a potential well loop
$\Lambda_{\tau(t),\sigma(t)}^{P_0,a_0}$
as in Lemma \ref{lem:pi1ofrelemistrivial}.
With the notation from there,
set
\[
\pi(t):=
a_0
\cdot
\big(1-\tau(t-3/4)\big)
\,.
\]
This time,
$a_0:=\min\big\{a_S,\varepsilon_S/2\big\}$
is chosen independently of the Hamiltonian.  
Invoking the Alexander dilation from
Section \ref{subsec:alexdilation},
observe that $\big(a_{\pi(t)}(\varphi_t),-\beta_{\pi(t)}(F_t)\big)$,
representing $A_{\pi(t)}(\Phi_t)$,
has support in
\[
\varphi^Y_{-\pi(t)}(V_S)=
V_S\setminus\Big(\big[-\pi(t),0\big)\times\partial V_S\Big)
\]
for all $t\in[0,1]$.

Define
\[
K(\Phi_t)_t:=
A_{\pi(t)}(\Phi_{\varrho(t)})\circ
\Lambda_{\tau(t),\sigma(t)}^{P_0(H_t),a_0}
\,.
\]
For $t\in[0,3/4]$ the function $\pi$ equals $a_0$,
so that the shape inequality follows from
Lemma \ref{lem:pi1ofrelemistrivial}.
For $t\in[3/4,1]$ the functions $\tau$, $\sigma$ and $\varrho$
equal $1$.
In this situation
$\Phi_{\varrho(t)}=(b-F_1,\varphi_1)$
and
$\Lambda_{\tau(t),\sigma(t)}^{P_0(H_t),a_0}=\id$,
for which the shape inequality is verified
in Lemma \ref{lem:alexshapeineq}.
As observed in the proof of
Proposition \ref{prop:homoequiv},
the dependence of the breathing lid $P_0(H_t)$
is continuous in $H_t$,
and hence in $\varphi$.
Therefore,
we get a time-$C^1$ contraction
\[
K\co
\MM(S)\times[0,1]\lra\MM(S)
\,,
\qquad
\big((\varphi,f),t\big)\longmapsto
K_t\big((\varphi,f)\big):=K(\Phi_t)_{1-t}
\,,
\]
so that $K_0=\id_{\MM(S)}$,
$K_1=(\id_{V_S},0)$ and
$K_t\big((\id_{V_S},0)\big)=(\id_{V_S},0)$
constantly for all $t\in[0,1]$.
\end{proof}


\section{Global parametrisation\label{sec:globparam}}

Let $(S,V_S,f_{\pm})$ be a sub- and super-homogeneous
Liouville shape in $(\R\times V,\rmd b+\lambda)$
whose shadow set $(V_S,\rmd\lambda)$
is a Hamiltonian contraction space.
We set
\[
\alpha_{\lambda}:=\rmd b+\lambda
\]
and write $\xi_{\lambda}:=\ker\alpha_{\lambda}$.


\subsection{\texorpdfstring{$U$}{U}-transform}
\label{subsec:utransform}

By Proposition \ref{prop:mscontractible},
there exists a time-$C^1$ contraction $K$ of $\MM(S)$.
As described at the beginning of
Section \ref{subsec:hamcontrspace},
the reversed contraction paths
\[
t\longmapsto K_{1-t}(\varphi,f)
\]
are Hamiltonian monodromy isotopies.
Their generating Hamiltonians assemble into
a continuous map
\[
H\co\MM(S)\times[0,1]
\lra C_c^{\infty}(V_S)
\]
such that
$H_t\big((\id_{V_S},0)\big)=0$
for all $t\in[0,1]$.
For every fixed $(\varphi,f)$,
the family $H_t(\varphi,f)$ is smooth in $t$.
Fixing a cut-off function $\chi$ uniformly
in Lemma \ref{lem:contifnotrappedorbitshamiltonianmonod}
and using the remark following its proof
defines the $\Psi$-{\bf transform},
a continuous map
\[
\Psi_{(\,.\,,\,.\,)}\co
\MM(S)\lra\Diff(\R\times V,\xi_{\lambda})
\,,
\]
by setting
$
\Psi_{(\varphi,f)}:=
\Psi_{(H_t(\varphi,f),\chi)}
$,
such that $\Psi_{(\id_{V_S},0)}=\id_{\R\times V}$.
Here,
$\Diff(\R\times V,\xi_{\lambda})$
denotes the group of contactomorphisms of
$(\R\times V,\xi_{\lambda})$
equipped with the compact-open
$C^\infty$-topology.
The map $\Psi_{(\,.\,,\,.\,)}$
comes with a unique continuous map
\[
h_{(\,.\,,\,.\,)}\co
\MM(S)\lra C_c^{\infty}\big(\Int(D_S)\big)
\,,
\]
given by
$
h_{(\varphi,f)}:=
h_{(H_t(\varphi,f),\chi)}
$,
such that $h_{(\id_{V_S},0)}=0$.
The map $h_{(\,.\,,\,.\,)}$
is determined by the Darboux chart
\[
\Psi_{(\varphi,f)}=\Phi_{\rme^h\alpha_{\lambda}}
\,,
\]
where $h=h_{(\varphi,f)}$ is the function appearing
in the conformal factor of the contact form $\rmd b+\lambda$
in Lemma \ref{lem:contifnotrappedorbitshamiltonianmonod}.
Observe that $\Psi_{(\,.\,,\,.\,)}$ and $h_{(\,.\,,\,.\,)}$
only depend on the class $[M,\alpha]$ in $\VV(S)$
under the identification with $(\varphi,f)$ in $\MM(S)$
via the monodromy bijection
from Proposition \ref{prop:monodromymap}.

The global Darboux chart satisfies
\[
(\Phi_{\alpha})_*\alpha_{\lambda}=\alpha
\]
for vertically convex contact forms,
where we simply write $\alpha$ instead of $\hat{\alpha}$.
We call the gluing-map-factorising diffeomorphism
  \[
  U_{\alpha}:=
  \Phi_{\rme^h\alpha_{\lambda}}\circ(\Phi_{\alpha})^{-1}
  \]
from Remark \ref{rem:contifnotrappedorbitshamiltonianmonod},
where $h=h_{(\varphi,f)}$ as above,
the $U$-{\bf transform}.
Notice that $U_{\alpha_{\lambda}}=\id_{\R\times V}$
and that
  \[
  (U_{\alpha})_*\alpha=
  \rme^{h_{(\varphi,f)}}\alpha_{\lambda}
  \,.
  \]
A change $\Phi\co(M',\alpha')\ra(M,\alpha)$
of the representative of $[M,\alpha]$ in $\VV(S)$
as in Section \ref{subsec:verticalcategory},
which necessarily equals
$\Phi=\Phi_{\alpha}\circ\Phi_{\alpha'}^{-1}$,
yields $U_{\alpha'}=U_{\alpha}\circ\Phi$.
Consequently,
we obtain that
$(U_{\alpha'})_*\alpha'=(U_{\alpha})_*\alpha$.

Let
  \[
  \AAA
  \]
be the space
of all vertically convex contact forms $\alpha$
on $\Int(D_S)$
such that $\alpha-\alpha_{\lambda}$ has compact support.
The subspace of $\xi_{\lambda}$-defining
contact forms in $\AAA$ is denoted by
$\AAA_{\xi_{\lambda}}$.
Observe that $\AAA_{\xi_{\lambda}}$ is contained in
\[
\rme^{C_c^{\infty}(\Int(D_S))}\cdot\alpha_{\lambda}
\,.
\]
The $U$-transform yields an injection
  \[
  \VV(S)
  \lra
  \AAA_{\xi_{\lambda}}
  \]
by sending
the class $[M,\alpha]\equiv(\varphi,f)$ 
to the contact form
$\rme^{h_{(\varphi,f)}}\alpha_{\lambda}$.
A left inverse is given by
\[
\II\co
\AAA_{\xi_{\lambda}}\lra\VV(S)
\,,
\quad
\rme^h\alpha_{\lambda}
\longmapsto
\big[D_S,\rme^h\alpha_{\lambda}\big]
\,,
\]
as the $U$-transform sends $[M,\alpha]$
to $\big[D_S,\rme^{h_{(\varphi,f)}}\alpha_{\lambda}\big]$,
see Remark \ref{rem:contifnotrappedorbitshamiltonianmonod}.

In order to pass to the corresponding orbit space,
consider the group
\[
\DD
\]
of compactly supported diffeomorphisms
of $\Int(D_S)$.
The subgroup of contactomorphisms
of $\xi_{\lambda}$ is denoted by
\[
\DD_{\xi_{\lambda}}=
\DD\cap\Diff(\R\times V,\xi_{\lambda})
\,.
\]
Setting
\[
  \VV_{\xi_{\lambda}}:=
  \AAA_{\xi_{\lambda}}/\DD_{\xi_{\lambda}}
  \,,
\]
the $U$-transform induces a map
  \[
  \VV(S)
  \lra
  \VV_{\xi_{\lambda}}
  \,,
  \]
which is a bijection.
The right inverse is given by applying
the above left inverse $\II$ to a representative
of $[\rme^h\alpha_{\lambda}]$.
Indeed,
because $(U_{\rme^h\alpha_{\lambda}})_*$ sends
$\rme^h\alpha_{\lambda}$
to the contact form
$\rme^{h_{(\varphi,f)}}\alpha_{\lambda}$,
where
$\big[D_S,\rme^h\alpha_{\lambda}\big]\equiv(\varphi,f)$,
we get $U_{\rme^h\alpha_{\lambda}}\in\DD_{\xi_{\lambda}}$,
i.e.\
$[\rme^{h_{(\varphi,f)}}\alpha_{\lambda}]=[\rme^h\alpha_{\lambda}]$.

\begin{rem}[Homeomorphic quotients]
\label{rem:homeoquspaces}
Define
  \[
    \VV:=\AAA/\DD
    \,.
  \]
By the above argument
the $U$-transform induces a homeomorphism
  \[
  \VV
  \lra
  \VV_{\xi_{\lambda}}
  \]
w.r.t.\ the quotient topology.
Its inverse is induced by the inclusion
$\AAA_{\xi_{\lambda}}\ra\AAA$.
\end{rem}


\subsection{Computing the quotient\label{subsec:compquot}}

Let $(S,V_S,f_{\pm})$
be a sub- and super-homogeneous Liouville shape
in $(\R\times V,\rmd b+\lambda)$
whose shadow set $(V_S,\rmd\lambda)$
is a Hamiltonian contraction space.
Then the vertical category $\VV(S)$ from
Section \ref{subsec:verticalcategory}
admits a more elementary description.
As seen in Section \ref{subsec:utransform}
and Remark \ref{rem:homeoquspaces},
the $U$-transform provides the quotient descriptions
$\VV$ and $\VV_{\xi_{\lambda}}$ of $\VV(S)$;
these quotient spaces are homeomorphic.
Using the $\Phi$-{\bf transform}
$(\Phi_{\alpha})_*\alpha_{\lambda}=\alpha$
defined by the global Darboux chart,
these quotient spaces can be described directly
in terms of the monodromy space
\[
\MM\equiv\MM(S)
\,.
\]
This gives the following topological variant of
Proposition \ref{prop:monodromymap}:

\begin{prop}[Monodromy homeomorphism]
\label{prop:monodromyhomeo}
  The $\Phi$-transform induces a homeomorphism
  \[
  \VV\lra\MM
  \,,
  \quad
  [\alpha]\longmapsto(\varphi_{\alpha},f_{\alpha})
  \,,
  \]
  such that
  $[\alpha_{\lambda}]\mapsto(\id_{V_S},0)$.
  The analogous statement with $\VV$ replaced by
  $\VV_{\xi_{\lambda}}\!$ holds.
\end{prop}

\begin{proof}
The map is defined by the restriction
$\Phi_{\alpha}|_{Z_{\alpha}}=(b+f_{\alpha},\varphi_{\alpha})$
on the super action range.
If $[\alpha]=[\alpha']$, then
$\Phi_{\alpha}\circ\Phi_{\alpha'}^{-1}\in\DD$
by Remark \ref{rem:freeactionstrnex}
and the first part of the proof of
Proposition \ref{prop:monodromymap}.
Consequently,
$Z_{\alpha}=Z_{\alpha'}$ and
$\Phi_{\alpha}$ and $\Phi_{\alpha'}$
agree on the super action range.
Hence,
$(\varphi_{\alpha},f_{\alpha})=(\varphi_{\alpha'},f_{\alpha'})$
and the monodromy map is well defined.

The inverse map is given via the $\Psi$-transform
  \[
  \MM\lra\VV
  \,,
  \quad
  (\varphi,f)\longmapsto
  [\rme^{h_{(\varphi,f)}}\alpha_{\lambda}]
  \,,
  \]
recalling the unique continuous map
$h_{(\,.\,,\,.\,)}\co\MM\ra C_c^{\infty}\big(\Int(D_S)\big)$
from Section \ref{subsec:utransform}
that satisfies  
$\Psi_{(\varphi,f)}=\Phi_{\rme^h\alpha_{\lambda}}$
with $h=h_{(\varphi,f)}$
and 
\[
(\Psi_{(\varphi,f)})_*\alpha_{\lambda}
=\rme^{h_{(\varphi,f)}}\alpha_{\lambda}
\quad\text{in}\quad
\AAA_{\xi_{\lambda}}
\,.
\]
Indeed,
applying the $\Psi$-transform to the image of
$[\alpha]\mapsto(\varphi_{\alpha},f_{\alpha})$
yields $[\rme^h\alpha_{\lambda}]=[\alpha]$
with $h=h_{(\varphi_{\alpha},f_{\alpha})}$,
because $(U_{\alpha})_*\alpha=\rme^h\alpha_{\lambda}$
and $U_{\alpha}\in\DD$.
Conversely,
the $\Phi$-transform applied to
$(\varphi,f)\mapsto[\rme^{h_{(\varphi,f)}}\alpha_{\lambda}]$
yields $(\varphi,f)$,
which is the monodromy of
$\Psi_{(\varphi,f)}=\Phi_{\rme^h\alpha_{\lambda}}$
with $h=h_{(\varphi,f)}$.
The first map is continuous
because the continuous monodromy map on $\AAA$
factors through the quotient projection $\AAA\ra\VV$.
The inverse map is continuous
by the continuity of $(\varphi,f)\mapsto h_{(\varphi,f)}$,
or equivalently of the $\Psi$-transform.
Thus, the two maps are mutually inverse homeomorphisms.

Finally,
$\Phi_{\alpha_{\lambda}}=\id_{\R\times V}$,
so that $[\alpha_{\lambda}]$ is mapped to
$(\id_{V_S},0)$.
The proof for $\VV_{\xi_{\lambda}}$ is identical.
\end{proof}


\subsection{Splitting\label{subsec:splitting}}

Assume that $(S,V_S,f_{\pm})$
is a sub- and super-homogeneous Liouville shape
in $(\R\times V,\rmd b+\lambda)$
whose shadow set $(V_S,\rmd\lambda)$
is a Hamiltonian contraction space.
Section \ref{subsec:compquot}
allows the following splitting of $\AAA$:

Identify $[\alpha]$ with the image
$(\varphi_{\alpha},f_{\alpha})$
under the monodromy homeomorphism
and write $h_{[\alpha]}$
for $h_{(\varphi_{\alpha},f_{\alpha})}$
and $\Psi_{[\alpha]}$
for $\Psi_{(\varphi_{\alpha},f_{\alpha})}$.
Because, as in the above proof,
the $U$-transform satisfies
$(U_{\alpha})_*\alpha=\rme^{h_{[\alpha]}}\alpha_{\lambda}$
and $U_{\alpha}\in\DD$,
the continuous projection
  \[
  \Pi\co
  \AAA\lra\VV
  \,,
  \quad
  \alpha\longmapsto[\alpha]
  \,,
  \]
admits a global continuous section
  \[
  \Sigma\co
  \VV\lra\AAA
  \,,
  \quad
  [\alpha]\longmapsto
  \rme^{h_{[\alpha]}}\alpha_{\lambda}
  \,,
  \]
induced by the $U$- or $\Psi$-transform.
Notice that
\[
\Sigma\big([\alpha_{\lambda}]\big)=\alpha_{\lambda}
\quad\text{as}\quad
h_{[\alpha_{\lambda}]}=0
\,.
\]
A section of the projection
$\AAA_{\xi_{\lambda}}\!\ra\VV_{\xi_{\lambda}}\!$ 
is obtained analogously.

\begin{lem}
\label{lem:freeaction}
In general,
the push-forward action
  \[
  \DD\times\AAA\lra\AAA
  \,,
  \quad
  (\Phi,\alpha)\longmapsto\Phi_*\alpha
  \,,
  \]
is free.
The analogous action of $\DD_{\xi_{\lambda}}\!$
on $\AAA_{\xi_{\lambda}}\!$ is also free.
\end{lem}

\begin{proof}
For a given fixed point $\alpha\in\AAA$
with $\Phi_*\alpha=\alpha$,
the $\Phi$-transform satisfies
$(\Phi_{\alpha})_*\alpha_{\lambda}=\alpha$.
Hence,
$(\Phi_{\alpha}^{-1}\circ\Phi\circ\Phi_{\alpha})_*\alpha_{\lambda}
=\alpha_{\lambda}$
with $\Phi_{\alpha}^{-1}\circ\Phi\circ\Phi_{\alpha}$
having support in $\Int(D_{T_{\alpha}})$.
By Remark \ref{rem:freeactionstrnex},
$\Phi_{\alpha}^{-1}\circ\Phi\circ\Phi_{\alpha}$
and hence $\Phi$ are equal to the identity.
\end{proof}

\begin{thm}[Equivariant splitting]
\label{thm:splitting}
Let $(S,V_S,f_{\pm})$
be a sub- and super-homo\-geneous Liouville shape
in $(\R\times V,\rmd b+\lambda)$
whose shadow set $(V_S,\rmd\lambda)$
is a Hamiltonian contraction space.

Then the push-forward of the section $\Sigma$,
\[
  \DD\times\VV\lra\AAA
  \,,
  \quad
  (\Phi,[\alpha])\longmapsto
  \Phi_*\Sigma\big([\alpha]\big)
  \,,
\]
is an $\DD$-equivariant homeomorphism
with inverse
\[
  \AAA\lra\DD\times\VV
  \,,
  \quad
  \alpha\longmapsto
  (U_{\alpha}^{-1},[\alpha])
  \,.
\]
It maps
$\big(\!\id_{\R\times V},[\alpha_{\lambda}]\big)$
to $\alpha_{\lambda}$.

  The analogous $\DD_{\xi_{\lambda}}$-equivariant
  statement for
  $\DD_{\xi_{\lambda}}\!\times\VV_{\xi_{\lambda}}
  \ra\AAA_{\xi_{\lambda}}\!$ also holds.
\end{thm}

\begin{proof}
The map is injective as the action is free:
Taking
$\Phi_*\Sigma\big([\alpha]\big)=\Phi_*'\Sigma\big([\alpha']\big)$,
an application of $\Pi$ yields
$\Pi\circ\Sigma\big([\alpha]\big)=\Pi\circ\Sigma\big([\alpha']\big)$
because $\Pi$ is constant along the orbits of the action.
Hence, $[\alpha]=[\alpha']$,
so that $\Phi=\Phi'$ by Lemma \ref{lem:freeaction}.
To see surjectivity, take $\alpha\in\AAA$.
Because
$[\alpha]=\Pi\circ\Sigma\big([\alpha]\big)$,
the forms $\alpha$ and $\Sigma\big([\alpha]\big)$
belong to the same $\DD$-orbit.
Hence, there exists $\Phi\in\DD$ such that
$\alpha=\Phi_*\Sigma\big([\alpha]\big)$.
Because $(U_{\alpha})_*\alpha$ and
$\Sigma\big([\alpha]\big)$ both equal
$\rme^{h_{[\alpha]}}\alpha_{\lambda}$,
we get $\Phi^{-1}=U_{\alpha}$ as the action is free.
Moreover,
we have $\Phi=U_{\Phi_*\Sigma([\alpha])}^{-1}$
for all $\Phi\in\DD$ and $[\alpha]\in\VV$,
so that the $\DD$-action
$(\Psi,\alpha)\mapsto\Psi_*\alpha$ on $\AAA$
transforms to the $\DD$-action
  $
  \big(\Psi,(\Phi,[\alpha])\big)
  \mapsto
  \big(\Psi\circ\Phi,[\alpha]\big)
  $
on $\DD\times\VV$.
The displayed map and its inverse are continuous,
since the push-forward action, the section $\Sigma$,
the projection $\Pi$, and the map
$\alpha\mapsto U_{\alpha}$ are continuous.
\end{proof}

\begin{rem}[Retraction]
\label{rem:retraction}
In view of the end of the proof
of Theorem \ref{thm:splitting},
\[
  \Phi=U_{\Phi_*\alpha}^{-1}
  \qquad\text{for all}\quad
  \Phi\in\DD
  \,\,\,\text{and}\,\,\,
  \alpha\in\Sigma(\VV)
  \,,
\]
so that $U_{\alpha}=\id_{\R\times V}$
for all $\alpha\in\Sigma(\VV)$ in particular.
Consequently, the $U$-transform
  \[
  \AAA\lra\AAA
  \,,
  \quad
  \alpha\longmapsto
  (U_{\alpha})_*\alpha
  \,,
  \]
whose image is
$\Sigma(\VV)\subset\AAA_{\xi_{\lambda}}$,
corresponds under the $\DD$-equivariant homeomorphism
of Theorem \ref{thm:splitting} to the retraction
  \[
  \DD\times\VV\lra\DD\times\VV
  \,,
  \quad
  (\Phi,[\alpha])\longmapsto
  \big(\!\id_{\R\times V},[\alpha]\big)
  \,.
  \]
The $\xi_{\lambda}$-variant of the correspondence
and Remark \ref{rem:homeoquspaces} yield
$\Sigma(\VV)=
\Sigma(\VV_{\xi_{\lambda}})\neq
\AAA_{\xi_{\lambda}}$.
\end{rem}

\begin{rem}[Non-representability]
\label{rem:non-representability}
Since $\MM$ is an open neighbourhood of
$(\id_{V_S},0)$ in $\EE$
and contains arbitrarily small nontrivial
Hamiltonian diffeomorphisms,
it is nontrivial.
Hence, $\VV\cong\MM$ is nontrivial,
cf.\ Proposition \ref{prop:monodromyhomeo}.
The restriction
  \[
  \DD_{\xi_{\lambda}}\times
  [\alpha_{\lambda}]
  \lra\AAA_{\xi_{\lambda}}
  \,,
  \quad
  \Phi\longmapsto
  \Phi_*\alpha_{\lambda}
  \,,
  \]
of the $\xi_{\lambda}$-variant
of the equivariant homeomorphism from
Theorem \ref{thm:splitting} cannot be surjective.
The homeomorphic image is equal
to the $\Pi$-fibre over $[\alpha_{\lambda}]$.
In other words,
there exists $h\in C_c^{\infty}\big(\Int(D_S)\big)$
for which the equation
$\Phi_*\alpha_{\lambda}=\rme^h\alpha_{\lambda}$
has {\it no} solution $\Phi$ in $\DD_{\xi_{\lambda}}$.
On the other hand,
observe that for
$\rme^h\alpha_{\lambda}\in\AAA_{\xi_{\lambda}}$
we have
$(\Phi_{\rme^h\alpha_{\lambda}})_*\alpha_{\lambda}=\rme^h\alpha_{\lambda}$.
The global Darboux chart
$\Phi_{\rme^h\alpha_{\lambda}}$,
however,
has compact support in $\Int(D_S)$,
meaning that
$\Phi_{\rme^h\alpha_{\lambda}}\in\DD_{\xi_{\lambda}}$,
if and only if $[\rme^h\alpha_{\lambda}]=[\alpha_{\lambda}]$,
i.e.\ in the case of trivial monodromy.

This orbit-theoretic viewpoint
is related to Polterovich's moduli space of contact forms
modulo the identity component of the contactomorphism group
\cite[Section~3]{pol02}.
In the present setting,
the monodromy provides an explicit complete parameter
transverse to the orbits.
\end{rem}

The preceding remark interprets the monodromy
as an obstruction to representability within the
$\DD$-orbit of $\alpha_{\lambda}$.
The following consequence of
Theorem \ref{thm:splitting}
shows that it provides a complete global parameter
transverse to the $\DD$-orbits.

\begin{cor}[Monodromy splitting]
\label{cor:monodromysplitting}
Let $(S,V_S,f_{\pm})$ be a sub- and super-homogeneous
Liouville shape in $(\R\times V,\rmd b+\lambda)$
whose shadow set $(V_S,\rmd\lambda)$
is a Hamiltonian contraction space.
The map
\[
\DD\times\MM\lra\AAA
\,,
\quad
\bigl(\Phi,(\varphi,f)\bigr)
\longmapsto
\Phi_*\bigl(
\rme^{h_{(\varphi,f)}}\alpha_{\lambda}
\bigr)
\,,
\]
which maps $\big(\id_{\R\times V},(\id_{V_S},0)\big)$
to $\alpha_{\lambda}$,
is an $\DD$-equivariant homeomorphism
with inverse
\[
\AAA\lra\DD\times\MM
\,,
\quad
\alpha\longmapsto
\bigl(
U_{\alpha}^{-1},
(\varphi_{\alpha},f_{\alpha})
\bigr)
\,.
\]
Likewise, the map
\[
\DD_{\xi_{\lambda}}\times\MM
\lra
\AAA_{\xi_{\lambda}}
\,,
\quad
\bigl(\Phi,(\varphi,f)\bigr)
\longmapsto
\Phi_*\bigl(
\rme^{h_{(\varphi,f)}}\alpha_{\lambda}
\bigr)
\,,
\]
is an $\DD_{\xi_{\lambda}}$-equivariant
homeomorphism.
\end{cor}

\begin{proof}
Under the monodromy homeomorphisms of
Proposition \ref{prop:monodromyhomeo},
the class corresponding to $(\varphi,f)$ is
$[\rme^{h_{(\varphi,f)}}\alpha_{\lambda}]$,
and its image under the section $\Sigma$
is $\rme^{h_{(\varphi,f)}}\alpha_{\lambda}$.
The assertions therefore follow directly from
Theorem \ref{thm:splitting}.
\end{proof}

\begin{cor}[Orbit deformation retracts]
\label{cor:orbitdeformationretracts}
Let $(S,V_S,f_{\pm})$ be a sub- and super-homogeneous
Liouville shape in $(\R\times V,\rmd b+\lambda)$
whose shadow set $(V_S,\rmd\lambda)$
is a Hamiltonian contraction space.
Then the orbits
$\DD\cdot\alpha_{\lambda}$ and
$\DD_{\xi_{\lambda}}\cdot\alpha_{\lambda}$
are strong deformation retracts of
$\AAA$ and $\AAA_{\xi_{\lambda}}$, resp.
In particular, the orbit maps
$\Phi\mapsto\Phi_*\alpha_{\lambda}$
induce homotopy equivalences
  \[
  \DD\simeq\AAA
  \quad\text{and}\quad
  \DD_{\xi_{\lambda}}\!\simeq\AAA_{\xi_{\lambda}}
  \,.
  \]
\end{cor}

\begin{proof}
Let $K$ be the strong deformation retraction of $\MM$
from Proposition \ref{prop:mscontractible}.
Under the monodromy splitting of
Corollary \ref{cor:monodromysplitting},
the orbit $\DD\cdot\alpha_{\lambda}$
corresponds precisely to
$\DD\times\{(\id_{V_S},0)\}$.
Hence, the product homotopy
$\id_{\DD}\times K$
is a strong deformation retraction of
$\DD\times\MM$ onto
$\DD\times\{(\id_{V_S},0)\}$.
Transporting this homotopy through the
monodromy splitting gives the asserted
strong deformation retraction of $\AAA$.
The argument for $\AAA_{\xi_{\lambda}}$
is identical.
\end{proof}

Examples can be found
in Section \ref{sec:spofcontforms} below.


\section{Hamiltonian monodromy}
\label{sec:hammonodrom}

Consider a vertically convex strict contact manifold $(M,\alpha)$
that is shaped by a strict vertically convex hypersurface
$(S,V_S,f_{\pm})$ in $(\R\times V,\rmd b+\lambda)$
via the gluing map
$\psi\co\big(U,\partial M,\alpha\big)\ra
\big(D_S,S,\rmd b+\lambda\big)$.
Assume that the first compactly supported de Rham cohomology
of the shadow set $V_S$ vanishes.
By \cite[Proposition~9.3.1]{mcdsal17},
the monodromy $\varphi_{\alpha}$ is Hamiltonian
precisely if $\varphi_{\alpha}$ lies in the connected component
of the identity in $\SS$.


\subsection{Vertically convex balls\label{sec:vcballs}}

The group $\SS$ is contractible
(and hence connected) if, for example,
the shadow set $\R^2_S$ is a
simply connected, bounded domain in $\R^2$,
see \cite[Theorem~1.8]{agz22}.
Gromov \cite{grom85} proved that
the group $\SS(\R^4_S)$ of all symplectomorphisms
of $(\R^4,\rmd\bfx\wedge\rmd\bfy)$
that have support in the closure of
an open, bounded, starshaped domain $\R^4_S$
with centre $0$ is contractible,
cf.\ \cite[Theorem~9.5.2]{mcsa12}.
In fact,
the proof given in \cite[Theorem~9.5.2]{mcsa12} shows
that the subgroup $\SS$ in $\SS(\R^4_S)$
taken w.r.t.\ the shadow set $\R^4_S$
is contractible.
Throughout the contraction described in
\cite[Theorem~9.5.2]{mcsa12},
which begins with an Alexander dilation,
the supports of the resulting symplectomorphisms
remain disjoint from the boundary $\partial\R^4_S$.

Equip $V=\R^{2n}$ with the primitive
\[
\lambda_{\st}=\tfrac12(\bfx\rmd\bfy-\bfy\rmd\bfx)
\]
of the standard symplectic form 
$\rmd\bfx\wedge\rmd\bfy$,
so that $\xi_{\st}=\xi_{\lambda_{\st}}$
is the standard contact structure on $\R\times\R^{2n}$.
Consider a shape
$(S,V_S,f_{\pm})$ in $(\R\times\R^{2n},\rmd b+\lambda_{\st})$
with $S$ diffeomorphic
to the unit sphere $S^{2n}=\partial D^{2n+1}$.
By the smooth generalised Sch\"onflies theorem \cite{sm61},
the domain $D_S$ is diffeomorphic to $D^{2n+1}$.
Assume that the shadow set $\R^{2n}_S$ is starshaped
with centre $0$ and require that
$f_{\pm}=0$ along its necessarily smooth boundary
$\partial\R^{2n}_S$.

For $n=1,2$,
the preceding contractibility results,
together with
$H_c^1(\R^{2n}_S)=0$,
show that every monodromy is Hamiltonian.
Consequently,
Theorem \ref{thm:monodromy} implies that
for any vertically convex contact manifold $(M,\xi)$
shaped by $(S,V_S,f_{\pm})$,
the gluing map of $(M,\xi)$ extends to a contactomorphism
$(M,\xi)\ra(D_S,\xi_{\st})$
after possibly restricting to a smaller neighbourhood
$U$ of $\partial M$,
cf.\ \cite[p.\ 1305, Item~1.b)]{eh94} and \cite[Theorem~1]{efz}.
In particular:

\begin{cor}
\label{cor:roundball}
 Let $n=1,2$.
 Consider a vertically convex contact manifold $(M,\xi)$
 shaped by $S^{2n}$ in $(\R\times\R^{2n},\rmd b+\lambda_{\st})$. 
 After possibly restricting to a smaller neighbourhood
 $U$ of $\partial M$,
 the gluing map of $(M,\xi)$ extends to a contactomorphism
 $(M,\xi)\ra(D^{2n+1},\xi_{\st})$.
 \qed
\end{cor}


\subsection{Spaces of contact forms\label{sec:spofcontforms}}

In fact,
the $0$-centred starshaped shadow sets $\R^{2n}_S$
described in Section \ref{sec:vcballs}
are examples of Hamiltonian contraction spaces
in the sense of Section \ref{subsec:hamcontrspace}
for $n=1,2$.
For simplicity we assume $S=S^{2n}$
as in Corollary \ref{cor:roundball}.

For $n=1$, the required time-$C^1$ property
from item (2) in Section \ref{subsec:hamcontrspace}
follows from the constructions of Smale \cite{sm59}
and \cite[Section~3]{agz22}.
Indeed,
in \cite[Lemma~1]{sm59},
the contraction is built
from a $C^{\infty}$-continuous family of vector fields
which is smooth in the deformation and space variables,
and in \cite[Lemma~2]{sm59} the corresponding diffeomorphisms
are obtained by integrating these vector fields.
The parametric Moser construction used in
the proof of \cite[Theorem~1.8]{agz22} preserves this regularity.
In particular,
the generating vector fields
of the resulting contraction of $\SS^2$
depend continuously on the initial diffeomorphism
in the strong $C^{\infty}$-topology.

For $n=2$,
the contraction of $\SS^4$
obtained from Gromov's proof \cite{grom85}
can be chosen smooth in all variables.
More precisely, every smooth family
of initial symplectomorphisms gives a family
of contraction paths that is smooth jointly
in the family parameter,
the contraction parameter and the space variable.
Indeed,
the relevant space of almost complex structures
admits a global chart onto a convex space,
and linear interpolation in this chart is smooth,
cf.\ \cite[p.~32]{gz23}.
All maps and homotopies used in passing between
the function spaces in Gromov's construction
are smooth as well,
cf.\ \cite[Lemma~9.5.6]{mcsa12},
which is part of the proof of
\cite[Theorem~9.5.2]{mcsa12}.
Together with the support observation in
Section \ref{sec:vcballs},
their finite compositions yield a contraction of $\SS^4$
that is smooth in all variables.
It is therefore time-$C^1$ in the sense of
Section \ref{subsec:hamcontrspace};
in particular, its generating vector fields
depend continuously on the initial symplectomorphism
in the strong $C^\infty$-topology.

It follows that
the homotopy equivalences
$\DD\simeq\AAA$ and
$\DD_{\xi_{\st}}\!\simeq\AAA_{\xi_{\st}}$
hold true for $n=1,2$
by Corollary \ref{cor:orbitdeformationretracts}.

\begin{proof}[{\bf Proof of Theorem \ref{thmintr:homotopytypeofaaa}}]
  The result is based on
  Corollary \ref{cor:orbitdeformationretracts}.
  With the preliminary remarks,
  the claim follows from the proof
  of the Smale conjecture
  by Hatcher \cite{hat83}.
  The necessary adaptations to $D^3$
  and the support conditions used
  to conclude that $\DD\simeq\{\id\}$
  in dimension $3$
  are worked out by Evers in den
  \cite[Hülfssätzen 1.2.2 und 2.1.2]{ev11}.
  The connectedness statement in dimension $5$
  follows analogously from
  $\pi_0(\DD^5)\cong\Gamma_6=0$;
  see Cerf \cite{cer70}
  and Kervaire--Milnor \cite{km63}.
\end{proof}

\begin{lem}[Relative Gray fibration]
\label{lem:relativegrayfibration}
Let $(N,\xi)$ be either $(B^3,\xi_{\st})$ or $(\R^3,\xi_{\st})$,
and denote by $\Diff_c(N)$ the group of compactly supported
diffeomorphisms of $N$
provided with the strong $C^{\infty}$-topology
and by $\operatorname{Cont}_c(N,\xi)$
the subgroup of contactomorphisms.
Consider the orbit $\OO_{\xi}:=\Diff_c(N)\cdot\xi$
of $\xi$ in the space of cooriented contact structures on $N$
that agree with $\xi$ outside a compact set.
We equip this space, and hence $\OO_{\xi}$,
with the strong $C^\infty$-topology
as a space of sections of the bundle
of cooriented hyperplanes in $TN$,
using the convention of
Section \ref{subsec:functionspaces}.
Then the orbit map
\[
o_{\xi}\co
\Diff_c(N)\lra\OO_{\xi}
\,,
\quad
\varphi\longmapsto\varphi_*\xi
\,,
\]
is a locally trivial fibre bundle
and hence a Serre fibration
whose fibre over $\xi$ is
$\operatorname{Cont}_c(N,\xi)$.
\end{lem}

\begin{proof}
Fix a Riemannian metric on $N$.
For every cooriented contact structure $\eta$
sufficiently close to $\xi$,
denote by $\alpha_{\eta}$
the positively cooriented unit $1$-form
with kernel $\eta$ determined by the metric.
The assignment $\eta\mapsto\alpha_{\eta}$
is continuous in the strong $C^\infty$-topology.
After restricting to a sufficiently small
neighbourhood $\UU$ of $\xi$,
the forms
$\alpha_{\eta,t}:=(1-t)\alpha_{\xi}+t\alpha_{\eta}$,
for $t\in[0,1]$,
are contact forms for all $\eta\in\UU$.
They agree with $\alpha_{\xi}$
wherever $\eta=\xi$.
The relative Gray vector field $X_{\eta,t}$
is determined by
$\alpha_{\eta,t}(X_{\eta,t})=0$ and
$\iota_{X_{\eta,t}}\rmd\alpha_{\eta,t}=-\dot\alpha_{\eta,t}$
on $\ker\alpha_{\eta,t}$,
cf.\ \cite[Theorem~2.2.2]{gei08}.
It vanishes wherever $\eta=\xi$
and is therefore compactly supported.
Its associated isotopy $\Phi_{\eta,t}$
exists for $t\in[0,1]$ and satisfies
$\Phi_{\eta,t}^*\ker\alpha_{\eta,t}=\xi$.
The defining equation and
standard continuous dependence
of the associated isotopies
on their time-dependent vector fields show that
$\Phi_{\eta,t}$ depends continuously on
$(\eta,t)$ in the strong $C^\infty$-topology.
Consequently, $\sigma(\eta):=\Phi_{\eta,1}$
defines a continuous local section of $o_{\xi}$
over $\UU$ with $\sigma(\xi)=\id$.
This section induces the local trivialisation
\[
o_{\xi}^{-1}(\UU)
\lra
\UU\times\operatorname{Cont}_c(N,\xi)
\,,
\quad
\Phi\longmapsto
\Big(
o_{\xi}(\Phi),
\sigma\big(o_{\xi}(\Phi)\big)^{-1}\circ\Phi
\Big)
\,.
\]
Its inverse is
\[
(\eta,\Psi)\longmapsto\sigma(\eta)\circ\Psi
\,.
\]
Translating this construction by elements of
$\Diff_c(N)$ gives local trivialisations
over the entire orbit $\OO_{\xi}$.
Hence, $o_{\xi}$ is a locally trivial fibration.
\end{proof}

\begin{proof}[{\bf Proof of Theorem
\ref{thmintr:homotopytypeofaaacont}}]
In view of Corollary \ref{cor:orbitdeformationretracts}
and the preceding remarks,
it remains to prove that
$\DD_{\xi_{\st}}$ is contractible for $n=1$.
Set $\CC:=\operatorname{Cont}_c(\R^3,\xi_{\st})$.
By the convention of Section \ref{subsec:functionspaces},
extension by the identity defines a continuous inclusion
$\DD_{\xi_{\st}}\ra\CC$.
On $\R^3$, the orbit $\OO_{\xi_{\st}}$
is the space of tight contact structures
that are standard at infinity.
It is weakly contractible by
Eliashberg--Mishachev
\cite[Theorem~1.1]{em}.
By Lemma \ref{lem:relativegrayfibration}
and the long exact homotopy sequence,
the inclusion $\CC\ra\Diff_c(\R^3)$
is a weak homotopy equivalence.
Hatcher's theorem \cite{hat83}
therefore implies that $\CC$ is weakly contractible.

Let $k\geq0$ and consider a continuous map
$\gamma\co S^k\ra\DD_{\xi_{\st}}$.
Regarded as a map into $\CC$,
it admits a null-homotopy
$\Gamma\co D^{k+1}\ra\CC$.
By compactness of $D^{k+1}$
and the defining locally finite neighbourhoods
of the strong $C^{\infty}$-topology,
the contactomorphisms in the image of $\Gamma$
have support in a common compact set $K\subset\R^3$.
For the standard primitive $\lambda_{\st}$,
the Liouville flow is complete.
Hence, the contact Alexander dilation from
Section \ref{subsec:alexdilation}
extends by conjugation to $\CC$
for all $t\geq0$.
Choose $a_0>0$ such that
$A_{a_0}\big(\Gamma(D^{k+1})\big)$
is supported in $B^3$.
Then
$A_{a_0}\circ\Gamma\co D^{k+1}\ra\DD_{\xi_{\st}}$
is a null-homotopy of $A_{a_0}\circ\gamma$.
On the other hand,
$A_t\circ\gamma$, for $t\in[0,a_0]$,
is a homotopy inside $\DD_{\xi_{\st}}$
from $\gamma$ to $A_{a_0}\circ\gamma$.
Consequently,
$\gamma$ is null-homotopic in $\DD_{\xi_{\st}}$,
and $\DD_{\xi_{\st}}$ is weakly contractible.

It remains to upgrade this statement.
By Lemma \ref{lem:relativegrayfibration},
the orbit map $o=o_{\xi_{\st}}$ of $\DD$
admits a continuous local section $\sigma$
near $\xi_{\st}$ with $\sigma(\xi_{\st})=\id$.
Let $\UU\subset\DD\cdot\xi_{\st}$
be an open neighbourhood of $\xi_{\st}$
on which the local section $\sigma$ is defined.
Then $o^{-1}(\UU)$ is an open neighbourhood
of $\DD_{\xi_{\st}}$ in $\DD$.
Define
\[
r\co
o^{-1}(\UU)\lra\DD_{\xi_{\st}}
\,,
\quad
\Phi\longmapsto
\sigma\big(o(\Phi)\big)^{-1}\circ\Phi
\,.
\]
This map is continuous,
since $o$, $\sigma$, inversion and composition
are continuous in the strong $C^\infty$-topology.
Setting $\eta:=o(\Phi)$, we obtain
$r(\Phi)_*\xi_{\st}=\sigma(\eta)^*\eta=\xi_{\st}$,
so that $r$ takes values in $\DD_{\xi_{\st}}$.
Moreover, if $\Phi\in\DD_{\xi_{\st}}$, then
$o(\Phi)=\xi_{\st}$ and hence $r(\Phi)=\Phi$.
Thus, $\DD_{\xi_{\st}}$ is a neighbourhood retract
of $\DD$.

The ANR argument used in the proof of
Proposition \ref{prop:homoequiv}
shows that $\DD$ is a metrisable ANR.
Hence, the open subset $o^{-1}(\UU)$ is an ANR,
and so is its retract $\DD_{\xi_{\st}}$.
Thus, $\DD_{\xi_{\st}}$ is a metrisable ANR
and has the homotopy type of a CW complex
by \cite[Theorem~2]{mi59}.
Its weak contractibility and Whitehead's theorem
therefore imply $\DD_{\xi_{\st}}\simeq*$.
\end{proof}


\section{Smooth case}
\label{sec:smoothcase}


\subsection{Odd-symplectic uniqueness
via vertical interpolation\label{subsec:oddsympuniness}}
Let $(M,\alpha)$ be a vertically convex
strict contact manifold
shaped by $(S,V_S,f_{\pm})$
in $(\R\times V,\rmd b+\lambda)$.
In view of Proposition \ref{prop:globalflowbox} (c),
we formulate the following.

\begin{lem}
\label{lem:verticallydiffeomorphic}
Consider a shape $(T,V_T,g_{\pm})$
in $(\R\times V,\rmd b+\lambda)$
equivalent to $(S,V_S,f_{\pm})$
with $g_-=f_-$.
Let $\chi\co\R\ra[0,1]$ be a smooth cut-off function
that is $0$ near $(-\infty,-1]$ and $1$ near $[0,\infty)$.

Then each such $\chi$ intrinsically determines
an odd-symplectomorphism
\[
\Delta_-\equiv
\Delta_{\chi}
\co
\big(\R\times V,D_T,\rmd\lambda\big)
\lra
\big(\R\times V,D_S,\rmd\lambda\big)
\]
of the form
\[
\Delta_-=(\delta,\id)
\qquad\text{with}\qquad
\delta_b>0
\,.
\]
We call $\Delta_\chi$ the
{\bf vertical interpolation diffeomorphism}.
The restriction of $\Delta_-$
to a neighbourhood of $Z_+=\{b\geq g_+(v)\}$
is equal to
 \[
   \Delta_-(b,v)=\Big(b+f_+(v)-g_+(v),v\Big)
 \]
for all $(b,v)$ in a neighbourhood of $Z_+$.
The restriction of $\Delta_-$ to a neighbourhood of
$\End(D_T)\setminus Z_+$ is the identity map.
\end{lem}

\begin{proof}
Consider the vertical isotopy
 \[
   \Delta_t(b,v)=\Big(b+t\big(f_+(v)-g_+(v)\big),v\Big)
 \]
for all $(b,v)\in\R\times V$ and $t\in[0,1]$.
The hypersurface $\Delta_t(T^+)$
is the graph of the function $g_+^t$ on $V_T$
given by convex interpolation $g_+^t:=tf_++(1-t)g_+$.
Define for all $t\in[0,1]$ and $(b,v)\in\R\times V_T$
\[
\tilde{\chi}_t(b,v):=
\chi\left(\frac{b-g_+^t(v)}{g_+^t(v)-f_-(v)}\right)
\,.
\]
As $g_+^t>f_-$ and
as the support of $f_+-g_+$ is contained in $V_T$
for all $t\in[0,1]$ we obtain a smooth family
of time-dependent vector fields
on $\R\times V$ by setting
\[
X_t:=\tilde{\chi}_t\!\cdot\!\big(f_+-g_+\big)\!\cdot\partial_b
\,.
\]
We get that $X_t$ equals
$\big(f_+-g_+\big)\!\cdot\partial_b$
in a neighbourhood of $\Delta_t(Z_+)$
and $0$ in a neighbourhood of
$\End(D_T)\setminus Z_+$.
Let $\tilde{\Delta}_t$ be the isotopy of $\R\times V$
generated by $X_t$.
As the integral curves are reparametrisations
of vertical lines $\R\times\{*\}$,
the isotopy is vertical, i.e.\ of the form
$\tilde{\Delta}_t=(\delta_t,\id)$, and, hence,
odd-symplectic w.r.t.\ $\rmd\lambda$.
Notice that $\tilde{\Delta}_t=\Delta_t$
in a neighbourhood of $Z_+$
and
$\tilde{\Delta}_t=\id$
in a neighbourhood of $\End(D_T)\setminus Z_+$.
Setting $\Delta_-:=\tilde{\Delta}_1$ yields the claim.
\end{proof}

We remark that the same construction,
without the assumption $g_-=f_-$,
yields a vertical diffeomorphism $\Delta'_-$
that interpolates the respective exit sets of the shapes.
By symmetry, a similar map $\Delta'_+$ for the entrance sets
is defined.
The composition $\Delta'_-\circ\Delta'_+$
with $g_+$ substituted by $g_++(f_--g_-)$
in the construction of $\Delta'_-$
yields a slightly different proof of
\cite[Lemma~4.1.1 and Remark~4.1.2]{efz}.

\begin{cor}[Odd-symplectomorphic]
 \label{cor:oddsymlifnotrappedorbits}
 Let $(M,\alpha)$ be a vertically convex strict contact manifold
 shaped by $(S,V_S,f_{\pm})$ in $(\R\times V,\rmd b+\lambda)$.
 Let $\chi\co\R\ra[0,1]$ be a smooth cut-off function
 that is $0$ near $(-\infty,-1]$ and $1$ near $[0,\infty)$.
 Let $\Phi_{\alpha}$ be the global Darboux chart 
 of $(M,\alpha)$ from Section \ref{subsec:globaldarbouxchart}
 and
 $\Delta_{\chi}$ be the vertical interpolation diffeomorphism
 from Lemma \ref{lem:verticallydiffeomorphic}
 taken w.r.t.\ the action domain of $\Phi_{\alpha}$.
 
 Then each such $\chi$ intrinsically determines
 an odd-symplectomorphism
 \[
  \Phi_{\alpha}\circ\Delta_{\chi}^{-1}\co
  \big(\R\times V,D_S,\rmd\lambda\big)
  \lra
  (\hat{M},M,\rmd\hat{\alpha})
  \]
 that preserves $\big(\!\End(D_S),\rmd\lambda\big)$
 and restricts to the identity on
  \[
  \End(D_S)\setminus
  \big\{(b,v)\in\R\times V_S\mid b\geq f_+(v)\big\}
  \,.
  \]
  For all $(b,v)\in\{b\geq f_+(v)\}$ we have
  \[
  \big(\Phi_{\alpha}\circ\Delta_{\chi}^{-1}\big)(b,v)=
  \Big(
    b+\big(\varphi_{\alpha}^*f_+-f_+\big)(v),\varphi_{\alpha}(v)
  \Big)
  \,.
  \]
In fact, the formula holds
for all $(b,v)$ in a neighbourhood of $\End(D_S)$,
where $\varphi_{\alpha}$ is replaced by $\id$
on a neighbourhood of
$\End(D_S)\setminus\{b\geq f_+(v)\}$.
\end{cor}

\begin{proof}
The statement follows by composing
the diffeomorphisms from
Proposition \ref{prop:globalflowbox} (c) and
Lemma \ref{lem:verticallydiffeomorphic}
as indicated.
\end{proof}

Consider the expression for
$\big(\Phi_{\alpha}\circ\Delta_{\chi}^{-1}\big)(b,v)$
for $(b,v)$ in a neighbourhood of $\{b\geq f_+(v)\}$.
Observe that the map
\[
\Phi_{\alpha}\circ\Delta_{\chi}^{-1}
=
\pi^{-1}\circ(\id\times\varphi_{\alpha})\circ\pi
\]
is the conjugation of
\[
(\id\times\varphi_{\alpha})(b,v)=\big(b,\varphi_{\alpha}(v)\big)
\]
by
\[
\pi(b,v):=\big(b-f_+(v),v\big)
\,.
\]
Here,
the map $\id\times\varphi_{\alpha}$
is defined on a neighbourhood of $[-b_0,\infty)\times V_S$
for a sufficiently small positive real number $b_0$.

\begin{prop}[Inward smooth gluing map extension]
\label{prop:inwardextension}
 Let $(M,\alpha)$ be a vertically convex strict contact manifold
 shaped by $(S,V_S,f_{\pm})$ in $(\R\times V,\rmd b+\lambda)$.
 Assume that the monodromy $\varphi_{\alpha}$ of $(M,\alpha)$
 is smoothly isotopic to the identity
 via diffeomorphisms $\varphi_t$
 with compact support in $V_S$.
 
 Then the gluing map $\psi$ of $(M,\alpha)$
 extends to a diffeomorphism $M\ra D_S$
 after restriction to a suitable neighbourhood
 of $\partial M$ if necessary.
 Conversely,
 if $\bar{V}_S$ is a smooth
 simply connected manifold with boundary and $n\geq3$,
 the existence of such an extension implies
 that $\varphi_{\alpha}$ is smoothly isotopic to the identity
 via diffeomorphisms with compact support in $V_S$.
\end{prop}

\begin{proof}
We continue the considerations
from Corollary \ref{cor:oddsymlifnotrappedorbits}.
Choose a smooth cut-off function $\tau$
that vanishes near $(-\infty,-b_0]$ and equals $1$ near $[0,\infty)$.
Replace $\id\times\varphi_{\alpha}$
by $(b,v)\mapsto\big(b,\varphi_{\tau(b)}(v)\big)$
in the above conjugation expression.
This defines a diffeomorphism
\[
\tilde{\Psi}\equiv\Psi^{\tau}
\]
of $\big(\R\times V,D_S\big)$
that coincides with $\Phi_{\alpha}\circ\Delta_{\chi}^{-1}$
in a neighbourhood of $\End(D_S)$.
Define $\Psi:=\Psi^{\tau}\circ\Delta_{\chi}$.
On a neighbourhood of $Z_+$, the map $\Psi$ is given by
\[
\begin{aligned}
\Psi
&=
\big(b+f_+,\id_V\!\big)
\circ
\big(b,\varphi_{\tau(b)}\big)
\circ
\big(b-g_+,\id_V\!\big)
\,.
\end{aligned}
\]
Indeed, on this neighbourhood,
$\Delta_{\chi}=\big(b+f_+-g_+,\id_V\!\big)$.
Thus, $\Psi$ agrees with $\Phi_{\alpha}$
near $Z_+$ and is the identity near
$\End(D_T)\setminus Z_+$.

Then
$\Psi\circ\Phi_{\alpha}^{-1}\co(\hat{M},M)\ra\big(\R\times V,D_S\big)$
is a diffeomorphism
equal to the identity in a neighbourhood of $\End(D_S)$.
After restriction to a smaller neighbourhood of the boundary,
$\Psi\circ\Phi_{\alpha}^{-1}$ is an extension
of the gluing map $\psi$ of $(M,\alpha)$
to a diffeomorphism $M\ra D_S$.

Conversely, assume these additional hypotheses.
Let $\Theta\co M\ra D_S$ be an extension of the gluing map.
Extending $\Theta$ over the attached ends by the identity
gives a diffeomorphism
$\hat{\Theta}\co\hat{M}\ra\R\times V$.
The composition
$\hat{\Theta}\circ\Phi_{\alpha}$
restricts to a diffeomorphism of $\R\times V_S$.
It is the identity for sufficiently negative $b$
and on a neighbourhood of $\R\times\partial V_S$.
For sufficiently positive $b$, it satisfies
\[
\hat{\Theta}\circ\Phi_{\alpha}(b,v)=
\big(b+f_{\alpha}(v),\varphi_{\alpha}(v)\big)
\,.
\]
This is the cylindrical description of a pseudo-isotopy
from \cite[Section~9.10]{ce12},
here relative to a neighbourhood of $\partial V_S$.

To remove the vertical translation,
set $h:=f_{\alpha}\circ\varphi_{\alpha}^{-1}$.
Choose a smooth cut-off function $\rho\co\R\ra[0,1]$
that is $0$ for sufficiently negative $b$
and $1$ for sufficiently positive $b$,
with $\|\rho'\|_{\infty}\|h\|_{\infty}<1$.
Then $\Lambda(b,v):=\big(b-\rho(b)h(v),v\big)$
is a diffeomorphism of $\R\times V_S$.
The composition $\Lambda\circ\hat{\Theta}\circ\Phi_{\alpha}$
is the identity near the negative and vertical ends
and equals $(b,v)\mapsto(b,\varphi_{\alpha}(v))$
near the positive end.
Restricting to $[-B,B]\times\bar{V}_S$
for sufficiently large $B>0$
and rescaling its first coordinate
therefore gives a pseudo-isotopy of $\bar{V}_S$
from the identity to $\varphi_{\alpha}$,
relative to a neighbourhood of $\partial V_S$.
Since $\bar{V}_S$ is simply connected
and $\dim V_S=2n\geq6$,
the relative form of Cerf's pseudo-isotopy theorem
\cite{cer70}, cf.\ \cite[Introduction]{kra22},
implies that $\varphi_{\alpha}$
is smoothly isotopic to the identity
relative to a neighbourhood of $\partial V_S$.
Restricting this isotopy to $V_S$
gives the required compactly supported isotopy.
\end{proof}

\begin{rem}[Odd-symplectic obstruction]
\label{rem:oddsympobstr}
The extension from Proposition \ref{prop:inwardextension}
is not odd-symplectic in general.
Indeed,
every odd-symplectomorphism of
$(\R\times V,\rmd\lambda)$ preserves
$\ker\rmd\lambda=\R\partial_b$.
For $\hat{\Psi}(b,v):=\bigl(b,\varphi_{\tau(b)}(v)\bigr)$
we have
\[
\hat{\Psi}_*\partial_b=
\partial_b+\tau'(b)\dot\varphi_{\tau(b)}(v)
\,.
\]
Thus, $\hat{\Psi}$ can be odd-symplectic only if
$\dot\varphi_{\tau(b)}=0$ whenever $\tau'(b)\neq0$.
In particular, this construction cannot interpolate
odd-symplectically to a non-trivial monodromy.
Conjugation by $\pi$ does not remove this obstruction,
since $\pi^*\rmd\lambda=\rmd\lambda$.

In fact, any odd-symplectic extension
$\Theta\co(M,\rmd\alpha)\ra(D_S,\rmd\lambda)$
of the gluing map forces $\varphi_{\alpha}=\id_V$.
Indeed, extending $\Theta$ over the attached ends
by the identity gives an odd-symplectomorphism
$\hat{\Theta}\co(\hat{M},\rmd\hat{\alpha})
\ra(\R\times V,\rmd\lambda)$.
The composition $\hat{\Theta}\circ\Phi_{\alpha}$
preserves the vertical line field
and is the identity near the negative end.
It therefore preserves each vertical line
$\R\times\{v\}$.
Near the positive end, this composition is given by
$(b,v)\mapsto(b+f_{\alpha}(v),\varphi_{\alpha}(v))$,
which implies $\varphi_{\alpha}=\id_V$.
\end{rem}

\begin{rem}
\label{rem:exotic}
If $S$ is diffeomorphic to the unit sphere
$S^{2n}=\partial D^{2n+1}$ in $\R\times\R^{2n}$,
so that, by the smooth Sch\"onflies theorem \cite{sm61}
in the dimensions under consideration,
$D_S$ is diffeomorphic to the unit disc $D^{2n+1}$,
a homeomorphic extension $\tilde{\Psi}$
is obtained with Alexander's trick.
If the group $\Gamma_{2n+1}$ is trivial,
a smooth extension $\tilde{\Psi}$ exists.
Using the uniqueness theorem for collar neighbourhoods,
the restriction of
$\Phi_{\alpha}\circ\Delta_{\chi}^{-1}$
to $\End(D_S)$ can be smoothly glued
to a diffeomorphism of $D_S$
whose restriction to the boundary $S$
agrees with
$\Phi_{\alpha}\circ\Delta_{\chi}^{-1}$.
\end{rem}

\begin{rem}
\label{rem:alternative}
In all dimensions,
the gluing map $\psi$
of a vertically convex contact manifold $(M,\xi)$
shaped by $S^{2n}$,
after a restriction to a smaller neighbourhood $U$ if necessary,
extends to a diffeomorphism $M\ra D^{2n+1}$
up to connected sum with a homotopy sphere,
see Remark \ref{rem:exotic}.
Whenever the group $\Gamma_{2n+1}$ 
of twisted oriented $(2n+1)$-spheres is trivial,
the push-forward of $\xi$
is a contact structure $\xi_{D}$ on $D^{2n+1}$
that coincides with $\xi_{\lambda}$ near the boundary.
If $\xi_{\lambda}=\xi_{\st}$,
the pushed-forward contact structure is tight
by Proposition \ref{prop:globalflowbox}.
By Eliashberg's classification \cite[Theorem~4.10.1(b)]{gei08}
(see \cite[Theorem~2.1.3]{elia92})
of tight contact structures on the $3$-ball,
the contact structure $\xi_{D}$ for $n=1$
is isotopic relative to a neighbourhood of $\partial D^3$
to $\xi_{\st}$.
Therefore, the gluing map $\psi$
extends to a contactomorphism
\[
(M,\xi)\lra(D^3,\xi_{\st})
\,.
\]
Corollary \ref{cor:roundball}
gives an alternative proof of the statement in
\cite[p.\ 1305, Item~1.b)]{eh94}.
\end{rem}


\subsection{Vertically convex vector fields\label{subsec:vertconvectfds}}

Vertical convexity of a contact form $\alpha$
is a property of its Reeb vector field $R_{\alpha}$.
If we ignore the contact form,
consider only {\it smooth} vector fields $X$,
and replace $\rmd b+\lambda$ by $\partial_b$
in the boundary condition in the definitions of 
Section \ref{subsec:shapedcontactmanifolds},
we obtain the analogous notion of
{\bf vertically convex vector fields} $X$
on manifolds $M$ shaped by $(S,V_S,f_{\pm})$
in $(\R\times V,\partial_b)$,
cf.\ \cite{buck26}.
No further structure on the
$m$-dimensional smooth manifold $V$ is assumed.

Vertical convexity of $(M,X)$
can be characterised by the existence
of a uniquely determined global flow-box chart,
the {\bf smooth} $\Phi$-{\bf transform},
\[
\Phi_X\co
\big(\R\times V,D_{T_X},\partial_b\big)
\lra
(\hat{M},M,\hat{X})
\,.
\]
Similarly to the contact case,
$\Phi_X$ satisfies the smooth properties
stated in Proposition \ref{prop:globalflowbox}.
In particular, $\Phi_X$ restricts to
\[
(b+f,\varphi)=
\big(b+f_+,\id_V\!\big)
\circ
(b,\varphi)
\circ
\big(b-g_+,\id_V\!\big)
\]
on $Z_+=Z_X$,
where the monodromy $\varphi=\varphi_X$ is an element
of the compactly supported diffeomorphism group
$\DD(V_S)$ of $V_S$.
Notice that $f=f_X$ is not determined by $\varphi_X$.
It records the {\bf flow-time}
\[
0<
\ell_X:=g_+-g_-=
(\varphi_X^*f_+-f_-)-f_X
\]
of the parametrised flow of $X$,
where $g_-=f_-$.

Denote by $\VV(S)^{\infty}$
the resulting {\bf smooth vertical category},
whose morphisms are diffeomorphisms
that send one vertical vector field to the other
and factorise the gluing maps.
Since morphisms in $\VV(S)^{\infty}$
conjugate the parametrised flows
and factorise the gluing maps,
the function $\ell_X$, and hence $f_X$,
is invariant under these morphisms.
The {\bf smooth monodromy space} $\MM^{\infty}(S)$
consists of all $(\varphi,f)\in\FF$
such that the shape inequality $f_-<\varphi^*f_+-f$ holds.
In contrast to the contact case,
no symplectic relation between $\varphi$ and $f$
is imposed.
The proof of Proposition \ref{prop:monodromymap}
and the smooth part of
Remark \ref{rem:freeactionstrnex}
give the analogue of the monodromy bijection
\[
\VV(S)^{\infty}\lra\MM^{\infty}(S)
\,,
\quad
[M,X]\longmapsto(\varphi_X,f_X)
\,.
\]

The proof of
Lemma \ref{lem:openshapeinequality}
applies verbatim to $\FF$ and
shows that $\MM^{\infty}(S)$ is open in $\FF$.
Restricting the projection of
$\FF=\DD(V_S)\times C_c^{\infty}(V_S)$
onto the first factor gives
\[
P\co
\MM^{\infty}(S)\lra\DD(V_S)
\,,
\quad
(\varphi,f)\longmapsto\varphi
\,.
\]
Its fibres $P^{-1}(\varphi)$
are convex and contain
$\big(\varphi,\varphi^*f_+-f_+\big)$
for every $\varphi\in\DD(V_S)$.
Notice, however, that the resulting
set-theoretic section
$\varphi\mapsto\big(\varphi,\varphi^*f_+-f_+\big)$
is continuous with respect to the strong $C^0$-topology,
but need not be continuous with respect to
the strong $C^{\infty}$-topology.
Because $\MM^{\infty}(S)$ is open in $\FF$,
for every $\varphi\in\DD(V_S)$
there exist a neighbourhood
$U_{\varphi}\subset\DD(V_S)$
and a function $f_{\varphi}\in C_c^\infty(V_S)$
such that $(\phi,f_{\varphi})\in\MM^{\infty}(S)$
for all $\phi\in U_{\varphi}$.
Since $\DD(V_S)$ is metrisable and hence paracompact,
there exist a locally finite refinement
$\{U_i\}$ of the open covering $\{U_{\varphi}\}$
and a continuous partition of unity
$\{\rho_i\}$ subordinate to $\{U_i\}$,
so that $\supp(\rho_i)\subset U_i$.
For every $i$, choose $\varphi_i$ such that
$U_i\subset U_{\varphi_i}$ and set
$f_i:=f_{\varphi_i}$.
If $\rho_i(\psi)\neq0$, then $(\psi,f_i)\in P^{-1}(\psi)$.
Local finiteness and convexity of the fibres therefore imply
that $z(\psi):=\sum_i\rho_i(\psi)f_i$
defines a continuous global section
\[
Z\co\DD(V_S)\lra\MM^{\infty}(S)
\,,
\quad
\varphi\longmapsto
\big(\varphi,z(\varphi)\big)
\,,
\]
of $P$.
Taking $f_{\id_{V_S}}=0$
and choosing the partition of unity
such that the corresponding function equals $1$
at $\id_{V_S}$,
we may assume that
$Z(\id_{V_S})=(\id_{V_S},0)$.
Convexity of the fibres implies that
$R_t(\varphi,f):=\big(\varphi,(1-t)f+t\,z(\varphi)\big)$
defines a time-$C^1$ strong deformation retraction
of $\MM^{\infty}(S)$ onto $Z\big(\DD(V_S)\big)$.
Indeed,
$\partial_tR_t(\varphi,f)=\big(0,z(\varphi)-f\big)$
depends continuously on $(\varphi,f,t)$.
Consequently,
\[
\MM^{\infty}(S)\simeq\DD(V_S)
\,.
\]

Using the vertical interpolation diffeomorphism
\[
\Delta_{\chi}=(\delta,\id)
\co
\big(\R\times V,D_{T_X},\partial_b\big)
\lra
\big(\R\times V,D_S,\delta_b\partial_b\big)
\]
from Lemma \ref{lem:verticallydiffeomorphic},
where $\delta=b+(f_+-g_+)$ on a neighbourhood of $Z_X$
and $\delta_b>0$ everywhere,
the odd-symplectic uniqueness statement from
Proposition \ref{prop:oddsymplnotnecext}
transfers to smooth uniqueness.
Namely, as in Corollary \ref{cor:oddsymlifnotrappedorbits},
there exists a diffeomorphism
  \[
  \Phi_X\circ\Delta_{\chi}^{-1}\co
  \big(\R\times V,D_S,\delta_b\partial_b\big)
  \lra
  (\hat{M},M,\hat{X})
  \,,
  \]
which restricts to the conjugation
$
\big(b+f_+,\id_V\!\big)
\circ
(b,\varphi_X)
\circ
\big(b-f_+,\id_V\!\big)
$
on a neighbourhood of $\{b\geq f_+\}$.
In order to obtain extendability
of restrictions of gluing maps,
Theorem \ref{thm:monodromy} can be replaced
by the construction from
Proposition \ref{prop:inwardextension}
whenever the monodromy is smoothly isotopic to the identity.
Recall $\pi(b,v)=\big(b-f_+(v),v\big)$
and define
\[
\hat{\Psi}_{\varphi}(b,v)
:=
\big(b,\varphi_{\tau(b)}(v)\big)
\,,
\]
cf.\ Remark \ref{rem:oddsympobstr}.
As in the proof of Proposition \ref{prop:inwardextension},
the conjugation
\[
\tilde{\Psi}_{\varphi}:=
\pi^{-1}\circ\hat{\Psi}_{\varphi}\circ\pi
\]
is a diffeomorphism of
$\big(\R\times V,D_S\big)$.

For an abstract smooth monodromy pair
$(\varphi,f)\in\MM^{\infty}(S)$,
denote by $T=T_{(\varphi,f)}$
the shape equivalent to $S$ determined by
$g_-:=f_-$ and
$g_+:=\varphi^*f_+-f$.
For the monodromy pair $(\varphi_X,f_X)$ of $X$,
we have $T_{(\varphi_X,f_X)}=T_X$.
Composing $\tilde{\Psi}_{\varphi}$
with the vertical interpolation diffeomorphism
associated with $T$
defines the {\bf smooth} $\Psi$-{\bf transform}
\[
\Psi_{(\varphi,f)}
:=
\tilde{\Psi}_{\varphi}\circ\Delta_{\chi}
\co
\big(\R\times V,D_T,\partial_b\big)
\lra
\big(\R\times V,D_S,X_{(\varphi,f)}\big)
\,,
\]
where
\[
X_{(\varphi,f)}
:=
\big(\Psi_{(\varphi,f)}\big)_*\partial_b
\,.
\]
On a neighbourhood of $Z=\{b\geq g_+\}$,
where $\Delta_{\chi}=(b+f_+-g_+,\id_V)$,
the map $\Psi_{(\varphi,f)}$ is given by
\[
\Psi_{(\varphi,f)}=
\big(b+f_+,\id_V\!\big)
\circ
\big(b,\varphi_{\tau(b)}\big)
\circ
\big(b-g_+,\id_V\!\big)
\,.
\]
Since $\tau=1$ on a neighbourhood of $[0,\infty)$,
it restricts to $(b+f,\varphi)$ on $Z$.
Near $\End(D_T)\setminus Z$
the map $\Psi_{(\varphi,f)}$ is the identity.
Consequently, for the monodromy pair
$(\varphi_X,f_X)$ of $X$,
the desired gluing-map-factorising diffeomorphism is
\[
\Psi_{(\varphi_X,f_X)}\circ\Phi_X^{-1}\co
(\hat{M},M,\hat{X})
\lra
\big(
\R\times V,
D_S,
X_{(\varphi_X,f_X)}
\big)
\,.
\]

Once the cut-off functions $\chi$
(from Lemma \ref{lem:verticallydiffeomorphic})
and
$\tau$
(from the proof of Proposition \ref{prop:inwardextension})
are fixed,
the map $\Psi_{(\varphi,f)}$ is determined by
the chosen isotopy $\varphi_t$ and the action domain
given by $g_+=\varphi^*f_+-f$.
The smooth variant of item (2) at the beginning of
Section \ref{subsec:hamcontrspace} 
yields the notion of a {\bf smooth contraction space} $V_S$.
For a smooth contraction space $V_S$,
a time-$C^1$ contraction of $\DD(V_S)$
provides a canonical and continuous choice
of the isotopies $\varphi_t$ similarly to
Section \ref{subsec:hamcontrspace}.
Consequently, the construction defines a continuous map
\[
\Psi_{(\,.\,,\,.\,)}
\co
\MM^{\infty}(S)
\lra
\Diff(\R\times V)
\,,
\]
such that $\Psi_{(\id_{V_S},0)}=\id_{\R\times V}$.
Here, the group of diffeomorphisms
$\Diff(\R\times V)$ of $\R\times V$
is equipped with the compact-open
$C^\infty$-topology.
Remark \ref{rem:freeactionstrnex} implies that
\[
\Psi_{(\varphi,f)}
=
\Phi_{X_{(\varphi,f)}}
\,.
\]
For $X$ with smooth monodromy $(\varphi_X,f_X)$,
we call
\[
U_X
:=
\Psi_{(\varphi_X,f_X)}\circ\Phi_X^{-1}
\]
the {\bf smooth} $U$-{\bf transform}.

Denote by $\XX$
the space of all vertically convex vector fields $X$
on $\Int(D_S)$ such that $X-\partial_b$
has compact support, and set
\[
\VV^{\infty}:=\XX/\DD
\,.
\]
The $U$-transform yields
a bijection
\[
\VV(S)^{\infty}\lra\VV^{\infty}
\,,
\quad
[M,X]\longmapsto
\big[X_{(\varphi_X,f_X)}\big]
\,,
\]
by the argument from
Section \ref{subsec:utransform}.
This yields a global parametrisation
of the vertical category as in the contact case.
Furthermore, the analogue of
Proposition \ref{prop:monodromyhomeo} holds true.
The $\Phi$-transform induces the
smooth monodromy homeomorphism
\[
\VV^{\infty}\lra\MM^{\infty}(S)
\,,
\quad
[X]\longmapsto(\varphi_X,f_X)
\,,
\]
such that
\[
[\partial_b]\longmapsto(\id_{V_S},0)
\,.
\]
Its inverse is induced by the $\Psi$-transform.

With the monodromy homeomorphism understood,
write
\[
\Psi_{[X]}:=\Psi_{(\varphi_X,f_X)}
\qquad\text{and}\qquad
X_{[X]}:=(\Psi_{[X]})_*\partial_b
\,.
\]
As in the considerations after
Proposition \ref{prop:monodromyhomeo},
the continuous projection
  \[
  \Pi\co
  \XX\lra\VV^{\infty}
  \,,
  \quad
  X\longmapsto[X]
  \,,
  \]
admits a global continuous section
  \[
  \Sigma\co
  \VV^{\infty}\lra\XX
  \,,
  \quad
  [X]\longmapsto X_{[X]}
  \,,
  \]
induced by the $U$- or $\Psi$-transform.
Notice that $\Sigma\big([\partial_b]\big)=\partial_b$.
Moreover,
the proof of Lemma \ref{lem:freeaction},
which is based on Remark \ref{rem:freeactionstrnex},
yields that, in general, the push-forward action
  \[
  \DD\times\XX\lra\XX
  \,,
  \quad
  (\Phi,X)\longmapsto\Phi_*X
  \,,
  \]
is free.
Therefore, the argument of
Theorem \ref{thm:splitting}
shows that the map
  \[
  \DD\times\VV^{\infty}\lra\XX
  \,,
  \quad
  (\Phi,[X])\longmapsto
  \Phi_*\Sigma\big([X]\big)
  \,,
  \]
is an $\DD$-equivariant homeomorphism satisfying
$\big(\!\id_{\R\times V},[\partial_b]\big)\mapsto\partial_b$.
The inverse is
  \[
  \XX\lra\DD\times\VV^{\infty}
  \,,
  \quad
  X\longmapsto(U_X^{-1},[X])
  \,.
  \]
  
Therefore,
if $V_S$ is a smooth contraction space,
the time-$C^1$ strong deformation retraction
of $\MM^{\infty}(S)$ onto
$Z\big(\DD(V_S)\big)$,
followed by the contraction of this image
induced by a time-$C^1$ contraction of $\DD(V_S)$,
gives a continuous strong deformation retraction
of $\MM^{\infty}(S)$ to $(\id_{V_S},0)$.
Although the first step is time-$C^1$,
the resulting contraction is, in general, only continuous,
since the section $Z$ is only known to be continuous.
Consequently,
\[
\DD
\,\,\simeq\,\,
\DD\times\DD(V_S)
\,\,\simeq\,\,
\DD\times\MM^{\infty}(S)
\,\,\cong\,\,
\DD\times\VV^{\infty}
\,\,\cong\,\,
\XX
\,.
\]
In fact, as in Corollary \ref{cor:orbitdeformationretracts},
the orbit $\DD\cdot\partial_b$
is a strong deformation retract of $\XX$,
and the homotopy equivalence
$\DD\simeq\XX$ is realised by the orbit map
\[
\DD\lra\XX
\,,
\quad
\Phi\longmapsto\Phi_*\partial_b
\,.
\]
  
For a smooth variant of Theorem \ref{thmintr:homotopytypeofaaa}
we consider $D_S=D^{m+1}$, so that $V_S=D^m$,
and write $\XX=\XX^{m+1}$ etc.
First of all,
contractibility of $\DD^m$ is equivalent
to the higher-dimensional analogue
of the Smale conjecture,
siehe \cite[Hülfssätze~1.2.2 und 2.1.2]{ev11}.
The only known examples
for which the contractions are time-$C^1$ are:
$\DD^1\simeq*$ via convex interpolation
and Smale's $\DD^2\simeq*$,
whose time-$C^1$ property we discussed in
Section \ref{sec:spofcontforms}.
Smale \cite{sm59} argues via
  \[
  \DD^2\,\,\simeq\,\,
  \DD^2\times\DD^1\,\,\simeq\,\,
  \XX^2
  \]
using $\DD^1\simeq*$ with the time-$C^1$
property as above
and an explicit time-$C^1$ contraction of $\XX^2$
that involves the Poincar\'e--Bendixson theorem.
The time-$C^1$ contraction $\DD^2\simeq*$
further yields
  \[
  \DD^3\,\,\simeq\,\,
  \DD^3\times\DD^2\,\,\simeq\,\,
  \XX^3
  \,,
  \]
which, together with Hatcher's \cite{hat83}
result $\DD^3\simeq*$, implies that:

\begin{thm}
\label{thm:homotopytypeofxx3}
$\XX^3\simeq*$
\end{thm}

It remains an open question whether
Hatcher's \cite{hat83} contraction of $\DD^3$
can be chosen time-$C^1$.
An affirmative answer would make $D^3$
a smooth contraction space and hence imply
$\DD^4\simeq\XX^4$.
This is particularly interesting in view of
the non-trivial rational homotopy groups of $\DD^4$
discussed by Botvinnik and Watanabe in \cite{bw23}.


\begin{ack}
  This work grew out of a series of lecture courses
  given at RUB in 2024--2026 and
  discussions in the A5/C5 Seminar.
  We would like to thank the members
  Franziska Beckschulte,
  Johanna Bimmermann,
  Jan Eyll,
  Jonas Fritsch,
  Marc Kegel,
  Lars Kelling,
  Bernd Stratmann,
  Manuel Stange
  and Anton Wilke
  for their contributions.
  We thank Alberto Abbondandolo,
  Barney Bramham,
  Hansjörg Geiges,
  Dominic Jänichen,
  Angeo Consympta von Querenburg,
  Stefan Suhr and
  Claudius Zibrowius
  for inspiring discussions.
\end{ack}


\end{document}